\documentclass[11pt, oneside]{article} 
\usepackage[margin=2cm]{geometry}                		
\usepackage[parfill]{parskip}    		
\usepackage{graphicx}			
\usepackage[utf8]{inputenc}
\usepackage[english]{babel} % English language/hyphenation
\usepackage{amsmath,amssymb,amsthm,bm} % Math packages
\usepackage[margin=2cm]{geometry}
\usepackage[color=yellow]{todonotes}
\usepackage[colorlinks]{hyperref}
\usepackage{enumerate}

\usepackage{pgfplots}
\usepackage{float}
\usepackage{tikz,tikz-cd}
\usepackage{graphicx}
\usepackage{caption}
\usepackage{multicol}
\usepackage{subcaption}
\usepackage{natbib}
\extrafloats{100}
\numberwithin{equation}{section}

\newtheorem{theorem}{Theorem}[section]
\newtheorem{lemma}[theorem]{Lemma}
\newtheorem{corollary}[theorem]{Corollary}
\newtheorem{proposition}[theorem]{Proposition}
\newtheorem{remark}[theorem]{Remark}
\newtheorem{example}[theorem]{Example}

\theoremstyle{definition}
\newtheorem{definition}[theorem]{Definition}

\pgfplotsset{compat=1.16}

\begin{document}

\title{\textbf{Symplectic techniques for stochastic differential equations on reductive Lie groups with applications to Langevin diffusions}}
\author{Erwin Luesink \thanks{\href{mailto:e.luesink@uva.nl} 
{e.luesink@uva.nl} Korteweg-de Vries Institute of Mathematics, University of Amsterdam} \and Oliver D. Street \thanks{\href{mailto:o.street18@imperial.ac.uk}{o.street18@imperial.ac.uk} Grantham Institute, Imperial College London} 
}
\date{}

\maketitle

\begin{abstract}
    We show how Langevin diffusions can be interpreted in the context of stochastic Hamiltonian systems with structure-preserving noise and dissipation on reductive Lie groups. Reductive Lie groups provide the setting in which the Lie group structure is compatible with Riemannian structures, via the existence of bi-invariant metrics. This structure allows for the explicit construction of Riemannian Brownian motion via symplectic techniques, which permits the study of Langevin diffusions with noise in the position coordinate as well as Langevin diffusions with noise in both momentum and position.
\end{abstract}

\section{Introduction}
In recent years, there has been growing interest in stochastic differential equations (SDEs) on Riemannian manifolds. They have applications in robotics, thermodynamics, material science, machine learning and spatial statistics. The key challenge with SDEs on Riemannian manifolds is the need for Riemannian Brownian motion (RBM), which is difficult to simulate in general. By specialising to reductive Lie groups, which are Riemannian manifolds with a compatible Lie group structure, we show that one can avoid the difficulties with RBM and obtain a framework for SDEs that is suitable for numerical simulation. Important examples of reductive Lie groups include spaces of invertible matrices $GL(n)$, spaces of unitary matrices $U(n), SU(n)$, spaces of symplectic matrices $Sp(n)$, and spaces of rotation matrices $SO(n)$.

\paragraph{Contributions of the paper.}
The present paper has two main contributions.
\begin{itemize}
    \item We demonstrate that the RBM -- and more generally, solutions to SDEs on reductive Lie groups -- can be sampled efficiently in a pathwise strong sense. That is, our approach avoids the use of geodesic random walks, which only converge to RBM in probability, and instead yields strong approximations of the continuous sample paths. This is explained in Section \ref{sec:eem}.
    \item Given a measure on a reductive Lie group, we employ symplectic techniques to derive families of Langevin equations as special cases of stochastic Hamiltonian systems with geometric dissipation on reductive Lie groups and use tools of symplectic geometry to prove that their ergodic invariant measure is the Gibbs measure. This is explained in Sections \ref{sec:mechanics} and \ref{sec:sampling}.
\end{itemize}

Combining our two major contributions, we pave the way for provably effective and efficient numerical schemes for sampling a measure on a reductive Lie group. We also emphasise the relations of kinetic Langevin equations to Hamiltonian mechanics, which explains the naturality of symplectic and energy-preserving (symmetric) integrators for the Hamiltonian parts. In combination with Lie-Trotter splitting, numerical integrators can be formulated that sample exactly from the invariant measure for kinetic Langevin diffusions, as shown in \cite{abdulle2014high}. 

In statistical mechanics, the invariant measure for deterministic Hamiltonian dynamics is the Gibbs measure, implying that deterministic Hamiltonian mechanics can be used to sample from measures. This insight underpins the Hamiltonian Monte Carlo method, introduced in \cite{duane1987hybrid} for sampling in lattice field theories. The Hamiltonian Monte Carlo method is a powerful and versatile sampling approach which is widely used in physics and beyond, but since the Hamiltonian is itself conserved under Hamiltonian mechanics and the dynamics is reversible, it may be slow to sample from the invariant measure. A different approach to sampling measures is by means of (kinetic) Langevin diffusions. Langevin diffusions are continuous-time stochastic processes governed by SDEs originally developed in physics where they model the motion of particles in a potential field connected to an external heat bath. 

Langevin dynamics can be considered on Riemannian manifolds, but we restrict to reductive Lie groups because their additional algebraic structure introduces powerful tools to construct RBM and techniques to analyse the SDEs. Numerical methods for (S)DEs on general Riemannian manifolds often require working in local charts, introducing complications such as step-size choices and chart transitions in the presence of noise. Differential equations on Lie groups can be solved without the need of local charts, but Lie groups, in general, lack canonical Riemannian metrics compatible with their group structure. This is where reductive Lie groups become particularly useful: they combine the structural advantages of Lie groups for numerical methods with the natural Riemannian metrics needed for diffusion processes. As a result, RBM enjoys a straightforward explicit construction on reductive Lie groups. By means of symplectic techniques, we construct a class of Langevin diffusions that contains the kinetic Langevin diffusion as a special case. The explicit access to RBM on reductive Lie groups allows for numerical discretisations which can be interpreted pathwise in contrast to methods that only provide stochastically weak approximations. 

By showing that Langevin diffusions on reductive Lie groups can be naturally formulated as a special class of stochastic Hamiltonian systems with geometric dissipation, we obtain a perspective that offers two key advantages. First, the coupling between noise and dissipation ensures that the invariant measure is always a Gibbs measure. This leads to a natural generalisation of kinetic Langevin diffusions, allowing for more flexible choices of noise while guaranteeing preservation of the target equilibrium distribution. Second, recasting the dynamics within the Hamiltonian framework enables and motivates the use of structure-preserving numerical integrators originally developed for deterministic Hamiltonian systems. Structure-preserving methods are known for their superior long-time accuracy and stability compared to non-structure-preserving schemes, and benefit from well-established pathwise error analysis. Moreover, this formulation allows for a unified treatment of Langevin-type diffusions with different noise structures. In particular, we distinguish between three variants: momentum Langevin, where noise is applied only to the momentum; position Langevin, where noise is applied only to the position; and symplectic Langevin, where noise affects both momentum and position. Each of these fits naturally into the stochastic Hamiltonian formalism, broadening the scope of Langevin dynamics on Lie groups.

Stochastic differential equations on manifolds is an active area of research, with several monographs discussing how to interpret solutions of such equations, see \cite{elworthy1998stochastic, hsu2002stochastic, ikeda2014stochastic}. SDEs on Riemannian manifolds and Lie groups have received growing attention from both theoretical and numerical perspectives. 

Such equations arise naturally in several fields. In physics and chemistry, for instance, active Brownian particles often exhibit rotational noise, as explored in \cite{kelidou2024active}. Similarly, machine learning on Lie groups frequently involves stochastic processes on curved spaces \cite{arnaudon2019irreversible, barbaresco2020lie}, with Riemannian stochastic gradient descent emerging as a powerful tool in optimization \cite{gess2024stochastic}. Langevin dynamics on curved geometries were first explored in the seminal work \cite{girolami2011riemann}. Recent advances in efficient numerical methods to solving stochastic differential equations based on geodesic random walks are discussed in works such as \cite{schwarz2023efficient, bharath2023sampling}. 

While a comprehensive overview is beyond the scope of this work, we note that Langevin dynamics in the Euclidean setting have been widely studied, with extensive theoretical and numerical analyses available; see for example \cite{kopec2015weakb, dalalyan2017theoretical, cheng2018convergence, vempala2019rapid, durmus2025uniform}. The kinetic Langevin process, in particular, has been the subject of detailed investigations in \cite{abdulle2014high, kopec2015weak, dalalyan2017theoretical, ma2021there, zhang2023improved, altschuler2024faster} and related works.

On Riemannian manifolds, Langevin diffusions have also been explored, typically using probabilistically weak approximations such as geometric random walks \cite{wang2020fast, cheng2022efficient, gatmiry2022convergence}. A recent and notable development is the detailed analysis of kinetic Langevin dynamics on Lie groups with left-invariant metrics \cite{kong2024convergence}, where a probabilistically strong approximation is used. In this work, we focus on SDEs on Lie groups equipped with bi-invariant metrics. While we restrict the geometry compared to \cite{kong2024convergence}, we generalize the class of diffusions considered by going beyond the kinetic Langevin formulation. Our approach contributes to the broader understanding of Langevin-type dynamics in non-Euclidean settings, where both geometry and stochasticity play essential roles.

\paragraph{Outlook.} In future work, we address the error analysis and the numerical implementation of the SDEs and Langevin diffusions discussed in the present work. We would also like to extend the results to reductive homogeneous spaces and discuss the setting for bi-invariant pseudo-metrics to incorporate curved geometries with negative curvature. We refer to Section \ref{sec:conclusion} for more details.

\paragraph{Overview of the paper.}
We provide a discussion of the necessary preliminaries on Lie groups and their relations to Riemannian manifolds in Section \ref{sec:preliminaries}, where we recall the seminal result of Cartan, explained in detail in \cite{milnor1976curvatures}, that proves the existence of bi-invariant metrics on reductive Lie groups. Using several results of the theory of reductive Lie groups, in Section \ref{sec:eem}, we construct RBM. We then introduce SDEs on smooth manifolds via Malliavin's transfer principle, which essentially states that sufficiently smooth manifold-valued curves can be replaced by manifold-valued semimartingales. In Section \ref{sec:mechanics}, we recall deterministic symplectic mechanics and the fundamental fact that deterministic symplectic dynamics have the Gibbs measure as their invariant measure. We then generalise this framework to stochastic symplectic mechanics via Malliavin's transfer principle and discuss that, even though all the geometric structure is maintained this way, the invariant measure changes. By introducing double-bracket dissipation, which is a geometry-preserving dissipation, we can cancel the effect of the noise on the invariant measure and re-obtain the Gibbs measure as the invariant measure, but associated to a system of SDEs. In Section \ref{sec:sampling}, we use the equations of stochastic geometric mechanics with so-called double-bracket dissipation to derive kinetic Langevin equations. Besides the usual kinetic Langevin equation, we also obtain a version that has noise only in the position and a version that has noise in both the momentum and the position. We illustrate the various kinds of Langevin dynamics through examples on reductive Lie groups and Euclidean space. In Section \ref{sec:conclusion}, we conclude and discuss future directions.

\section{Preliminaries on Lie groups}\label{sec:preliminaries}
In general, Riemannian manifolds and Lie groups are fundamentally different objects. A Riemannian manifold is a smooth manifold with a smoothly varying inner product defined at each point and a Lie group is a smooth manifold with a group structure whose actions are smooth. In certain situations, a Riemannian manifold can be equipped with a group structure which is compatible with the Riemannian structure or, vice versa, a Lie group can be equipped with a Riemannian metric which is compatible with the group structure. The necessary notion is that of a bi-invariant metric. To avoid confusion with regards to notation, we denote a Riemannian metric on a Lie group $G$ by $\langle\,\cdot\,,\,\cdot\,\rangle$ and reserve $g$ for denoting group elements unless otherwise explicitly specified. Bi-invariant metrics are denoted by $Q(\,\cdot\,,\,\cdot\,)$ and we denote the Killing form by $\kappa(\,\cdot\,,\,\cdot\,)$. Before discussing the conditions for existence of bi-invariant metrics in detail, we first give some preliminary definitions and discuss the intuition behind the necessity of bi-invariant metrics.

\subsection{Preliminary definitions}
In the study of Lie groups and differential equations, Lie algebras play a central role. This is because the tangent space at any point of a Lie group can be right-translated to the identity of the group, and the tangent space at the identity of the group is the Lie algebra. The natural operations of the Lie group on its Lie algebra lead to the adjoint representations and the natural operations of the Lie group on the dual of its Lie algebra provide the coadjoint representations. For more background on Lie groups and Lie algebras, we refer to \cite{rossmann2006lie, hall2013lie} and for more information on Riemannian geometry, we refer to \cite{lee2018introduction}. In the paragraph below, we recall the basic definitions for matrix Lie groups.

Given a Lie group, one has the smooth left-translation $L_g$ which, for $g,h\in G$, maps $h\mapsto gh$ and the smooth right-translation $R_g$ which, for $g,h\in G$, maps $h\mapsto hg$. Neither of these maps have fixed points, which makes their study complicated. We consider the inner automorphism $\Psi:G\to \mathrm{Aut}(G)$, given by $\Psi_g(h):= L_g\circ R_{g^{-1}}(h) = ghg^{-1}$, which has the group identity as a fixed point. The differential of $\Psi_g$ defines the action of the Lie group on its Lie algebra as $\mathrm{Ad}:G\to\mathrm{Aut}(\mathfrak{g})$ by $\mathrm{Ad}_g(X) := (d\Psi)_g(X) = g Xg^{-1}$, where $g\in G$ and $X\in\mathfrak{g}$. The map $\mathrm{Ad}:G\to\mathrm{Aut}(\mathfrak{g})$ is called the adjoint representation of the group. Taking another differential leads to the adjoint representation of the algebra $\mathrm{ad}:\mathfrak{g}\to\mathrm{End}(\mathfrak{g})$, defined by $\mathrm{ad}_X(Y) := (d\mathrm{Ad})_X(Y) = [X,Y] = XY-YX$, where $X,Y\in\mathfrak{g}$. Since a Lie algebra is a vector space, one can consider its dual space. The natural pairing for matrix Lie algebras is the Frobenius pairing (or Hilbert-Schmidt pairing) given by $\langle X,Y\rangle_{HS} = \mathrm{tr}(X^\dagger Y)$, where $\dagger$ denotes the Hermitian adjoint, which reduces to the transpose if the Lie algebra is defined over the field of real numbers. The coadjoint representation of the group $\mathrm{Ad}^*:G\to\mathrm{Aut}(\mathfrak{g}^*)$ is defined as the dual operator to the adjoint representation of the group $\mathrm{Ad}:G\to\mathrm{Aut}(\mathfrak{g})$ with a small twist, $\langle\mathrm{Ad}^*_{g^{-1}} \mu,X\rangle_{HS} := \langle\mu, \mathrm{Ad}_{g}X\rangle_{HS}$, where $g\in G$, $X\in\mathfrak{g}$ and $\mu\in\mathfrak{g}^*$. The additional inverse is necessary for $\mathrm{Ad}^*$ to be a left representation. The differential of the coadjoint representation of the group leads to the coadjoint representation of the algebra $\mathrm{ad}^*:\mathfrak{g}\to\mathrm{End}(\mathfrak{g}^*)$, defined by $\mathrm{ad}^*_X(\mu) = (d\mathrm{Ad}^*)_X(\mu)$, where $X\in\mathfrak{g}$ and $\mu\in\mathfrak{g}^*$. Alternatively, one can define the coadjoint representation of the algebra by duality as follows $\langle -\mathrm{ad}^*_X\mu, Y\rangle_{HS} = \langle \mu, \mathrm{ad}_XY\rangle_{HS}$, where the minus sign follows from the inverse used in the duality relation for the representations of the group.

One of the fundamental ideas that we use repeatedly in this work is the fact that a connected Lie group can be reconstructed from a neighbourhood of its identity. The Lie algebra captures the local structure of the group around the idenity element. The Lie exponential map $\exp:\mathfrak{g}\to G$ is a diffeomorphism in this neighbourhood of the identity and is defined, for $X\in\mathfrak{g}$, as $\exp(X)=\gamma(1)$ with $\gamma:\mathbb{R}\to G$ the unique one-parameter subgroup of $G$ whose tangent vector at the identity is $X$. The Lie exponential together with the adjoint and coadjoint representations satisfy the following identities
\begin{equation}\label{eq:liegroupliealgebra}
    \begin{aligned}
        \Psi_{g}(\exp X) &= \exp(\mathrm{Ad}_gX),\\
        \mathrm{Ad}_{\exp X} &= \exp\mathrm{ad}_X,\\
        \mathrm{Ad}^*_{\exp X} &= \exp \mathrm{ad}^*_X,
    \end{aligned}
\end{equation}
where $g\in G$ and $X\in\mathfrak{g}$. Note that, if the exponential map fails to be surjective, these identities hold only locally in a neighbourhood of the identity. For connected, simply connected Lie groups, these identies are valid globally, since the exponential map is a diffeomorphism in that case. The second and the third expressions are particularly important for our purposes. The proofs of these identities can be found in \cite{rossmann2006lie} and \cite{hall2013lie}. Next, we consider the preliminary definitions with regards to the intersection of Riemannian geometry with Lie theory.

\begin{definition}
    A Riemannian metric $\langle\,\cdot\,,\,\cdot\,\rangle$ on a Lie group $G$ is left-invariant if $L_g$ is an isometry for all $g\in G$. That is, if we have
    \begin{equation}
        \langle d(L_g)_hX, d(L_g)_h Y\rangle_{gh} = \langle X,Y\rangle_h \,,
    \end{equation}
    for all $g,h\in G$ and $X,Y\in T_hG$. Similarly, a Riemannian metric is right-invariant if $R_g$ is an isometry for all $g\in G$. 
\end{definition}

Given an inner product $\langle\,\cdot\,,\,\cdot\,\rangle$ on the Lie algebra $\mathfrak{g}\simeq T_eG$, one can always define a left-invariant metric on $G$ by setting
\begin{equation}
    \langle X,Y\rangle_g := \langle d(L_{g^{-1}})_gX, d(L_{g^{-1}})_gY\rangle_e \,,
\end{equation}
for all $g\in G$ and $X,Y\in T_gG$ and, again, the right-invariant case is analogous. Of particular importance to us is the notion of a bi-invariant metric.

\begin{definition}
    A bi-invariant metric $Q$ on a Lie group $G$ is a Riemannian metric that is both left- and right-invariant.
\end{definition}

The following two lemmas go back to the original work of Cartan and determine when a given left-invariant Riemannian metric is bi-invariant. A modern treatment can be found in \cite{milnor1976curvatures} These lemmas use the link between Lie groups and Lie algebras.

\begin{lemma}\label{lemma:adinvariantgroup}
    A left-invariant metric on a Lie group $G$ is also right-invariant if and only if, for each group element $g\in G$, the linear transformation $\mathrm{Ad}_g:\mathfrak{g}\to\mathfrak{g}$ is an isometry with respect to the induced metric on the Lie algebra $\mathfrak{g}$.
\end{lemma}

\begin{proof}
    The adjoint action of the group on its Lie algebra $\mathrm{Ad}_g$ is defined as the differential of the inner automorphism $\Psi_g = L_gR_{g^{-1}}$, i.e.,  $\mathrm{Ad}_g = d(L_gR_{g^{-1}})$. If the metric $\langle\,\cdot\,,\,\cdot\,\rangle$ is left-invariant, then $L_g$ is an isometry. If the metric is also right-invariant, then $R_g$ is also an isometry, meaning that $\langle d(L_g R_{g^{-1}})_eX, d(L_g R_{g^{-1}})_eY\rangle = \langle X,Y\rangle_e$ holds for all $g,h\in G$ and $X,Y\in T_eG\simeq\mathfrak{g}$. Conversely, if $\mathrm{Ad}_g$ is an isometry for every $g\in G$, then $L_g R_{g^{-1}}$ is an isometry and the result follows. 
\end{proof}

\begin{lemma}\label{lemma:adinvariantalgebra}
    A left-invariant metric on a connected Lie group is bi-invariant if and only if $\mathrm{ad}_X:\mathfrak{g}\to\mathfrak{g}$ is skew-adjoint for every $X\in \mathfrak{g}$.
\end{lemma}

\begin{proof}
    For $g\in G$ sufficiently close to the identity, there exists a unique $X\in\mathfrak{g}$ such that $g=\exp X$. The adjoint action of the group is an orthogonal transformation if $\mathrm{Ad}_{g^{-1}} = \mathrm{Ad}^*_g$. Since $\mathrm{Ad}_g = \mathrm{Ad}_{\exp X} = \exp\mathrm{ad}_X$, we obtain $\exp\mathrm{ad}^*_X = \mathrm{Ad}^*_g = \mathrm{Ad}_{g^{-1}} = \exp (-\mathrm{ad}_X)$, from which we deduce that $\mathrm{ad}^*_X = -\mathrm{ad}_X$. A connected Lie group is generated by any neighbourhood of the identity and compositions of orthogonal transformations are orthogonal, so the proof is complete.
\end{proof}

The following theorem, stated and proved in \cite{milnor1976curvatures}, makes precise when a Lie group admits a bi-invariant metric. The proof of this result is also treated in detail \cite{alexandrino2015lie} and in \cite{gallier2020differential}.
\begin{theorem}\label{thm:Milnor}
    A connected Lie group $G$ admits a bi-invariant metric if and only if it is reductive, that is, if $G$ is isomorphic to the direct product of a compact group and an abelian group.
\end{theorem}
Recall that an abelian group is a commutative group. The proof of this theorem uses many concepts and results that we use also in the development of stochastic differential equations and their analytical or numerical solution later in this work. Hence, we repeat the proof not only for completeness, but also as a means to introduce the concepts of Lie theory and Riemannian geometry on which we will later rely. The first step of the proof of Theorem \ref{thm:Milnor} is to show the if direction, which requires results on unimodular and compact Lie groups that we discuss in the next section.

\subsection{Unimodular and compact Lie groups}\label{sec:compactliegroups}
The key concept in this section is the Haar measure. The Haar measure is defined in general as follows.
\begin{definition}
    Let $G$ be a locally compact topological group. The left and right Haar measures on $G$ are non-negative measures $\mu_L$ and $\mu_R$, respectively, on the Borel sets on $G$ which satisfy the following invariance property: for any Borel set $A\subset G$ and any $g\in G$, we have
    \begin{equation}
        \mu_L(L_gA) = \mu_L(A), \qquad \mu_R(R_gA)=\mu_R(A).
    \end{equation}
\end{definition}
In \cite{federer1969geometric}, one can find the following important results. The left and right Haar measures are constant multiples of one another. Groups for which the left and right Haar measures coincide are called unimodular groups. Unimodularity is necessary for the existence of bi-invariant measures. Given Lie group structure rather than topological structure, one considers unimodular Lie groups. In this work, we focus on bi-invariant metrics, which leads us to consider compact Lie groups, abelian Lie groups and their direct product. 

Compact Lie groups allow for analytical approaches to study their geometry. Integration on general smooth manifolds is understood in the Riemannan-Stieltjes sense. The Haar measure extends Riemann-Stieltjes integration to Lebesgue integration on compact Lie groups. An explicit construction of the Haar measure uses a left-invariant volume form on a Lie group.

A volume form $\omega$ on a Lie group $G$ is called left-invariant if $L_g^*\,\omega = \omega$ and right-invariant if $R_g^*\,\omega = \omega$. On any Lie group there is, up to scalar multiplication, a unique left-invariant volume form. This follows from the fact that the dimension of the space of volume forms at the identity in the Lie group is equal to one, since then up to multiplication by a nonzero scalar, there is a unique choice of $\omega_e\in \Omega^n(G)_e$. Defining $\omega_g = L^*_{g^{-1}}\omega_e$ then gives us a left-invariant volume form on $G$. If $G$ is compact then, up to multiplication by $\pm 1$, there is a unique left-invariant volume form $\omega$ on $G$ such that $\int_G \omega=1$. To see this, recall that replacing the volume form $\omega$ by $c\omega$ for some nonzero scalar $c$ multiplies the value of the integral by $|c|$. Since $G$ is compact, $\int_G \omega$ is finite with respect to the volume form. Hence there is a unique $c$ up to a multiplication by $\pm 1$, so that $\int_G \omega=1$ with respect to the volume form. 

Given a compact Lie group $G$, let $\omega$ be a left-invariant volume form that is normalised so that $\int_G \omega=1$. For any $f\in C(G)$, define
\begin{equation}
    \int_G f\mu(dg) =\int_Gf|\omega|
\end{equation}
with respect to the orientation given by $\omega$. By means of the Riesz representation theorem, $|\omega|$ is the completion of the volume form to a Borel measure on $G$. This measure $\mu$ is the Haar measure on a compact Lie group and has several important properties.

\begin{lemma}
    Let $G$ be compact. The Haar measure $\mu$ is left-invariant, right-invariant and invariant under inversion:
    \begin{equation}
        \int_Gf(hg)\mu(dg) = \int_Gf(gh)\mu(dg) = \int_Gf(g^{-1})\mu(dg) =\int_Gf(g)\mu(dg)
    \end{equation}
    for any $h\in G$ and $f$ any Borel measurable function on $G$. 
\end{lemma}
For a proof we refer to \cite{sepanski2007compact}. The Haar measure is used to obtain bi-invariant metrics from general Riemannian metrics through averaging, which works as follows. Suppose that $\omega$ is a left-invariant volume form on $G$. To make it bi-invariant, we can define a new volume form $\widetilde{\omega}$ by averaging $\omega$ over the group using the Haar measure
\begin{equation}
    \widetilde{\omega} = \int_GR_g^*\omega \,\mu(dg).
\end{equation}
In a similar manner, a given Riemannian metric on $G$ that is not bi-invariant can be averaged to make it bi-invariant. The following lemma shows that it is straightforward to obtain bi-invariant metrics on compact Lie groups.
\begin{lemma}\label{lemma:compactliegroup}
    A compact Lie group $G$ admits a bi-invariant metric.
\end{lemma}
\begin{proof}
    Let $\omega$ be a right-invariant volume form on $G$ and let $\langle\,\cdot\,,\,\cdot\,\rangle$ denote a right-invariant metric on $G$. We define for all $X,Y\in T_x G$
    \begin{equation}
        Q(X,Y)_x:= \int_G\langle dL_g X, dL_g Y\rangle_{gx} \omega.
    \end{equation}
    We first show that $Q$ is left-invariant. Fix $X,Y\in T_g G$ and consider the function $f:G\to \mathbb{R}$ given by $f:G\to\mathbb{R}$ given by $f(g)=\langle dL_g X, dL_g Y\rangle_{gx}$. A quick computation then shows that
    \begin{equation}
    \begin{aligned}
        Q(dL_hX,dL_hY)_{hx} &= \int_G\langle dL_g(dL_h X), dL_g(dL_hY)\rangle_{g(hx)}\omega = \int_G\langle dL_{gh}X, dL_{gh}Y\rangle_{(gh)x} \omega\\
        &= \int_G f(gh)\omega =\int_G R^*_h(f(g)\omega) = \int_G f(g)\omega = \int_G\langle dL_g X,dL_gY\rangle_{gx} \omega = Q(X,Y)_x
    \end{aligned}
    \end{equation}
    where we have used the right-invariance of the metric as well as the right-invariance of the volume form. This shows that $Q$ is left-invariant. The following computation shows that $Q$ is also right-invariant
    \begin{equation}
        \begin{aligned}
            Q(dR_hX,dR_h Y)_{xh} &= \int_G\langle dL_g(dR_h X), dL_g(dR_h Y)\rangle_{g(xh)}\omega = \int_G\langle dR_h(dL_g X),dR_h(dL_gY)\rangle_{(gx)h}\omega\\
            &= \int_G \langle dL_g X,dL_g Y\rangle_{gx}\omega = Q(X,Y)_x .
        \end{aligned}
    \end{equation}
    This finishes the proof.
\end{proof}

\begin{proof}[Proof of ``if" in Theorem \ref{thm:Milnor}.]
    Any abelian group admits a bi-invariant measure since left- and right-translations are the same. By Lemma \ref{lemma:compactliegroup}, any compact Lie group admits a bi-invariant metric. Hence the direct product between a compact Lie group and an abelian group admits a bi-invariant metric.
\end{proof}

For the converse, additional algebraic insights are required. In the next section, we first discuss Lie groups that admit nondegenerate quadratic forms and proceed to semisimple Lie groups, which provide an algebraic condition for compactness.

\subsection{Quadratic, reductive and semisimple Lie groups}\label{sec:quadraticliegroups}
The goal of this section is twofold. Firstly, we introduce the necessary statements to decide when a structure on a Lie algebra is enough to imply a bi-invariant metric on the Lie group. Secondly, we discuss the Killing form and its role in determining whether a Lie group is compact. Before discussing some of the relevant notions regarding semisimplicity, we first introduce quadratic Lie algebras. Quadratic Lie algebras were first introduced in \cite{tsou1957xix}, under the name metrisable Lie algebras. This is precisely the setting for the present work. 
\begin{definition}
    A Lie algebra $\mathfrak{g}$ is called quadratic if it admits a nondegenerate symmetric bilinear form $\varphi:\mathfrak{g}\times\mathfrak{g}\to\mathbb{R}$ that is invariant under the adjoint action. That is, for all $X,Y,Z\in\mathfrak{g}$, we have
    \begin{equation}
        \varphi([X,Y],Z) + \varphi(Y,[X,Z]) = 0.
    \end{equation}
    A Lie group is called quadratic if its Lie algebra is quadratic.
\end{definition}
The presence of such a nondegenerate symmetric bilinear form implies that $\mathrm{ad}^*_X = -\mathrm{ad}_X$ for all $X\in\mathfrak{g}$. The quadratic Lie algebras constitute the largest class of Lie algebras for which this property holds. Lemma \ref{lemma:adinvariantalgebra} requires a left-invariant metric, which is a stronger condition than a left-invariant bilinear form due to the lack of definiteness. Classifying the quadratic Lie algebras is an open problem, see \cite{ovando2016lie, conti2021uniqueness}. A large class of quadratic Lie algebras are the semisimple ones, which can be characterised by their Killing form, which is defined as follows.

\begin{definition}
    The Killing form on a Lie algebra $\mathfrak{g}$ is the symmetric bilinear form $\kappa:\mathfrak{g}\times\mathfrak{g}\to\mathbb{R}$ defined by
    \begin{equation}
        \kappa(X,Y) = \mathrm{tr}(\mathrm{ad}_X\circ\mathrm{ad}_Y).
    \end{equation}
    A Lie algebra is called semisimple if and only if its Killing form is nondegenerate and a Lie group is called semisimple if its Lie algebra is semisimple.
\end{definition}

For a semisimple Lie algebra, the Killing form is a nondegenerate symmetric bilinear form, hence every semisimple Lie algebra is quadratic. One of the convenient properties of the Killing form on semisimple Lie groups is that it provides a way of checking for compactness, as shown by the following lemma.
\begin{lemma}
    A semisimple connected Lie group $G$ is compact if and only if its Killing form is negative definite.
\end{lemma}
For the proof we refer to \cite{alexandrino2015lie}. It is now clear that compactness is sufficient for the existence of bi-invariant metrics and semisimple connected Lie groups provide an algebraic condition for compactness but, as Theorem \ref{thm:Milnor} indicates, a slightly more general setting is possible.

From Lemma \ref{lemma:adinvariantgroup}, we have that adjoint action of the group $G$ is an element of the orthogonal group $O(\mathfrak{g})$ defined over the Lie algebra $\mathfrak{g}$, i.e., $\mathrm{Ad}(G) = \{\mathrm{Ad}_g\,|\, \forall g\in G\}\subset O(\mathfrak{g})$. Since the orthogonal group is compact, the subgroup $\mathrm{Ad}(G)$ has compact closure in $O(\mathfrak{g})$. Using the averaging trick that was used in the proof of Lemma \ref{lemma:compactliegroup}, it can be shown that this condition is also sufficient. By applying the following theorem (Theorem 5.2 in \cite{sternberg1999lectures}) to the adjoint representation, Lemma \ref{lemma:compactliegroup} is generalised.
\begin{theorem}\label{thm:sternberg}
    Let $\rho:G\to GL(V)$ be a finite-dimensional representation of a Lie group $G$. There is a $G$-invariant inner product on the vector space $V$ if and only if $\rho(G)$ has compact closure in $O(V)$.
\end{theorem}

\begin{corollary}\label{cor:compactclosure}
    For any Lie group $G$, an inner product on $\mathfrak{g}$ induces a bi-invariant metric on $G$ if and only if $\mathrm{Ad}(G)$ has compact closure in $O(\mathfrak{g})$.
\end{corollary}

The proof of Theorem \ref{thm:sternberg} uses the averaging trick that was also used in the proof of Lemma \ref{lemma:compactliegroup}. It immediately follows that if $G$ is itself compact then $\mathrm{Ad}(G)$ has compact closure, since the adjoint representation is then the image of a compact set by a continuous mapping. For abelian Lie groups, the result is also immediate since the adjoint action reduces to the identity. In general, one sees that the existence of bi-invariant metrics on Lie groups is determined by the properties of the adjoint representation. One may define reductive Lie groups precisely as Lie groups that satisfy Corollary \ref{cor:compactclosure}. This explains the intuition behind Theorem \ref{thm:Milnor}. In \cite{arsigny2006bi}, the result above is used to elegantly show that the special Euclidean group $SE(n)$, being the semidirect product of a compact Lie group and an abelian group, admits no bi-invariant metric for $n\geq 2$, since the adjoint representation does not have compact closure. 

To complete the proof of Theorem \ref{thm:Milnor}, it needs to be shown that the universal covering group of a connected Lie group with a bi-invariant metric splits into the direct product of a compact Lie group and an abelian group. The universal covering group shows up to avoid putting the simply-connected restriction on the Lie group. To show that the universal covering group indeed splits as described, we first discuss several important properties of bi-invariant metrics.

\subsection{Properties of bi-invariant metrics}
In this section we discuss a number of important algebraic and geometric properties of bi-invariant metrics. 
\subsubsection{Algebraic properties}
An ideal $I$ of a Lie algebra $\mathfrak{g}$ is a Lie subalgebra $I\subset\mathfrak{g}$ that is closed under the Lie bracket, i.e., $[I,\mathfrak{g}]\subseteq I$. A Lie algebra is called simple if it contains no ideals other than the trivial ideal and itself. It immediately follows that any simple abelian ideal is one-dimensional. The center of a Lie algebra $Z(\mathfrak{g})$ is the subset $Z(\mathfrak{g})\subset\mathfrak{g}$ that contains elements of the Lie algebra that commute with every other element, i.e., $Z(\mathfrak{g}) = \{ X\in \mathfrak{g}\,|\, [X,\mathfrak{g}] = 0\}$.

\begin{lemma}\label{lemma:liealgebradecomp}
If a Lie algebra is equipped with a bi-invariant metric, then the orthogonal complement of any ideal is also an ideal, and the Lie algebra decomposes as the direct sum of orthogonal simple ideals.
\end{lemma}
\begin{proof}
    Let $I\subset\mathfrak{g}$ be an ideal and let $Q(\,\cdot\,,\,\cdot\,):\mathfrak{g}\times\mathfrak{g}\to\mathbb{R}$ be the bi-invariant metric. If $Y\in\mathfrak{g}$ is orthogonal to $I$, then $[X,Y]$ should be orthogonal to $I$. For any $Z\in I$, we have $\langle [X,Y],Z\rangle = -\langle Y,[X,Z]\rangle =0$ by the $\mathrm{ad}$-invariance of the metric. Hence $\mathfrak{g}=I\oplus I^\perp$, where $I^\perp$ is the orthogonal complement of $I$. Repeating this argument until only simple ideals remain, the result follows. 
\end{proof}

Thus, by Lemma \ref{lemma:liealgebradecomp}, a Lie algebra with a bi-invariant metric decomposes as
\begin{equation}\label{eq:liealgdecomp}
    \mathfrak{g} = I_1\oplus \cdots \oplus I_n,
\end{equation}
where each $I_i$ is a simple ideal. For every Lie algebra it holds that its Killing form restricted to an ideal is the Killing form on the ideal, which implies that the Killing form on a Lie algebra is the sum of the Killing forms of its ideals. Proofs of these statements can be found in \cite{jacobson1979lie}. Each ideal corresponds to a connected subgroup of the universal covering group and the converse also holds. The simply-connected Lie group $\widetilde{G}$ is the universal cover of $G$ and has $\mathfrak{g}$ as its Lie algebra. Hence, the decomposition of $\mathfrak{g}$ implies a decomposition of $\widetilde{G}$
\begin{equation}
    \widetilde{G} = A_1\times \cdots \times A_n
\end{equation}
where each $A_i$ is a normal subgroup. Now there are two cases. Either $I_i$ is abelian, which means that it is one-dimensional and the normal subgroup $A_i$ is isomorphic to $\mathbb{R}$, or $I_i$ is noncommutative with trivial center. If every ideal is of this second kind, i.e., noncommutative and simple, then the Lie algebra $\mathfrak{g}$ is semisimple. Hence the decomposition \eqref{eq:liealgdecomp} can be formulated as $\mathfrak{g}=\mathfrak{s}\oplus \mathfrak{a}$, where $\mathfrak{s}$ is semisimple (therefore itself the direct sum of simple ideals) and $\mathfrak{a}$ is abelian (itself the direct sum of one-dimensional ideals). It is natural to introduce the following definition at this stage.
\begin{definition}
    A Lie algebra $\mathfrak{g}$ is called reductive if it decomposes as
    \begin{equation}
        \mathfrak{g} = \mathfrak{s}\oplus\mathfrak{a},
    \end{equation}
    where $\mathfrak{s}$ is semisimple and $\mathfrak{a}$ is abelian.
\end{definition}
Note that this is not the Levi decomposition, which states that every Lie algebra can be decomposed into a semisimple part and a solvable radical. Solvable radicals are not necessarily abelian. 

If $I_i$ is noncommutative and simple, the proof of Theorem \ref{thm:Milnor} is complete if the corresponding normal subgroup $A_i$ is compact. To conclude this, a geometric property of bi-invariant metrics is required: the Ricci curvature is strictly positive. We recall this fact in the next section. In addition, we also obtain that $\mathrm{ad}$-invariant inner products on reductive Lie algebras correspond to bi-invariant metrics on reductive Lie groups and this is precisely the setting in which it is possible to obtain such metrics.

\subsubsection{Geometric properties}
Bi-invariant metrics have the following geometric properties.
\begin{lemma}
Let $(G,Q)$ be a Lie group with bi-invariant metric $Q(\,\cdot\,,\,\cdot\,)$, then one has the following explicit formulas for the covariant derivative, the curvature tensor and the sectional curvature
\begin{align}
\nabla_XY &= \frac{1}{2}[X,Y] \label{eq:covariantclg},\\
R(X,Y)Z &= -\frac{1}{4}[[X,Y],Z] \label{eq:curvatureclg},\\
Q(R(X,Y)Y,X) &= \frac{1}{4}\|[X,Y]\|^2 \label{eq:sectionalclg},\\
\mathrm{Ric}(X,Y) &=-\frac{1}{4}\kappa(X,Y). \label{eq:ricciclg}
\end{align}
\end{lemma}

\begin{proof}
The first identity follows from the Koszul formula (also known as the fundamental theorem of Riemannian geometry, see \cite{helgason1979differential}) as follows. For any metric $\langle\,\cdot\,,\,\cdot\,\rangle$, the Koszul formula is given by
\begin{equation}\label{eq:koszul}
    \langle\nabla_XY,Z\rangle = \frac{1}{2}\left(X\langle Y,Z\rangle + Y\langle Z,X\rangle - Z\langle X,Y\rangle + \langle[X,Y],Z\rangle - \langle[X,Z],Y\rangle -\langle[Y,Z],X\rangle\right).
\end{equation}
Since a bi-invariant metric $Q$ is left-invariant by definition, we have $XQ(Y,Z) = YQ(Z,X) = ZQ(X,Y) = 0$. The invariance of $Q$ under the adjoint action of the Lie algebra then implies the result. For the second identity, we compute the curvature tensor directly from its definition and use the previous result to obtain
\begin{equation}
    R(X,Y)Z = \nabla_X\nabla_YZ - \nabla_Y\nabla_X Z - \nabla_{[X,Y]}Z = -\frac{1}{4}[[X,Y],Z].
\end{equation}
For the third identity \eqref{eq:sectionalclg}, we compute directly using the second identity
\begin{equation}
    Q(R(X,Y)Y,X) = -\frac{1}{4}Q([[X,Y],Y],X) = \frac{1}{4}\|[X,Y]\|^2.
\end{equation}
The fourth identity also follows from a direct computation using the previously derived identities
\begin{equation}
    \mathrm{Ric}(X,Y) = \sum_Z Q(R(Z,X)Y,Z) = -\frac{1}{4}\mathrm{tr}[X,[Y, \,\cdot\,]] =  -\frac{1}{4}\kappa(X,Y).
\end{equation}
This completes the proof.
\end{proof}

Note that none of the identities in \eqref{eq:covariantclg}-\eqref{eq:sectionalclg} depend on the bi-invariant metric explicitly. The mere presence of a bi-invariant metric is enough to induce algebraic structure in the Riemannian setting. Since the Ricci tensor is proportional to the Killing form, on compact semisimple Lie algebras, where the Killing form is negative definite, one obtains that the Ricci tensor is strictly positive. This is an essential step in the proof of Theorem \ref{thm:Milnor}, since the Bonnet-Myers theorem implies that a semisimple Lie group with minus the Killing form as its metric is a compact and complete Riemannian manifold. This allows us to finish the proof.

\begin{proof}[Proof of ``only if" in Theorem \ref{thm:Milnor}]
The Killing form of a noncommutative simple ideal $I_i$ induces a Ricci tensor that is strictly positive on the corresponding normal subgroup $A_i$. Hence, the Bonnet-Myers theorem then implies that $A_i$ compact. To complete the proof of Theorem \ref{thm:Milnor}, all that remains is to show that if a Lie group $G$ admits a bi-invariant metric, then its universal covering group $\widetilde{G}$ splits into a compact group $H$ and an abelian group $\mathbb{R}^m$. The group $G$ can be identified with the quotient $\widetilde{G}/\Pi$, where $\Pi$ is a discrete normal subgroup. By projecting $\Pi$ into $\mathbb{R}^m$, one can define the vector space $V$ as the image of the projection and the orthogonal complement $V^\perp$. Then $G=(H\times V)/\Pi \oplus V^\perp$, which is the direct sum of a compact Lie group and an abelian group.
\end{proof}

Next we discuss geodesics on Lie groups with bi-invariant metrics. Using the explicit form of the covariant derivative, it can be deduced that for Lie groups with bi-invariant metrics the Christoffel symbols are proportional to the structure constants $C_{ij}^k$ of the Lie algebra. Let $\{X_i\}_{i=1}^{\dim\mathfrak{g}}$ be a basis for $\mathfrak{g}$, then
\begin{equation}\label{eq:christoffelclg}
\nabla_{X_i}X_j = \frac{1}{2}[X_i,X_j] = \frac{1}{2}\sum_{k=1}^{\dim \mathfrak{g}} C_{ij}^k X_k.    
\end{equation}
On compact semisimple Lie groups with minus the Killing form as the metric, the Ricci curvature is a scalar multiple of the metric, which implies that compact semisimple Lie groups are Einstein manifolds. The scalar is explicitly known and equal to $\frac{1}{4}$, which follows from \eqref{eq:ricciclg}.

Given the formula for the covariant derivative \eqref{eq:covariantclg}, we can compute the geodesic equation in Lie algebraic terms. A geodesic $\gamma:\mathbb{R}\to G$ is a curve on the Lie group that minimises the metric. This means that it must satisfy the differential equation
\begin{equation}
    \nabla_{\dot{\gamma}}\dot{\gamma} = 0,
\end{equation}
where $\nabla$ denotes the Levi-Civita connection associated with the metric and $\dot{\gamma}$ is the derivative of the curve $\gamma$. With respect to the bi-invariant metric, we have the relation $\nabla_X Y = \frac{1}{2}[X,Y]$, meaning that a geodesic satisfies
\begin{equation}
    \frac{1}{2}[\dot{\gamma},\dot{\gamma}] = 0.
\end{equation}
For a geodesic initially at $e\in G$, we can make the ansatz $\gamma(t)=\exp tX$ with $X\in\mathfrak{g}$ and $\exp:\mathfrak{g}\to G$ the Lie exponential map. Substituting this in the equation above, we obtain, using the fact that $\exp tX\in G$,
\begin{equation}
    [\dot{\gamma},\dot{\gamma}] = [X\exp tX, X\exp tX] = [X,X]\exp tX = 0,
\end{equation}
so we see that $\gamma(t)=\exp tX$ satisfies the geodesic equation for any $X\in\mathfrak{g}$. For a geodesic initially at a different point $g_0\in G$, we right translate by $R_{{g_0}^{-1}}$ to map $g_0$ to the identity. The velocity vector $X$ at $g_0$ is then transformed into $(R_{g_0^{-1}})_*X$. The geodesic is then given by $g(t)=\exp(t (R_{g_0^{-1}})_*X)$. Right translating the geodesic back to the original initial data then gives the geodesic starting at $g_0$ as $g(t)=g_0\exp(t (R_{g_0^{-1}})_*X)$. This shows that geodesics on a Lie group with a bi-invariant metric are one-parameter subgroups of the Lie group. This partially proves the following theorem that we use extensively.
\begin{theorem}\label{thm:lieexp}
    The Lie exponential map and the Riemannian exponential map at the identity agree on Lie groups endowed with bi-invariant metrics. The Lie exponential map of a compact connected Lie group is surjective.
\end{theorem}

\begin{proof}
    The first part was proved by the computation above. To see that the exponential map of a compact and connected Lie group $G$ is surjective, note that the Lie exponential is defined for all $X\in\mathfrak{g}$. The compactness of the Lie group implies that there exists a bi-invariant metric $Q$. Given the connectedness of $G$, it then follows from the Hopf-Rinow theorem that $(G,Q)$ is a complete Riemannian manifold, which implies that $\exp = \exp_e:T_eG\to G$ is surjective.
\end{proof}

In the next section, we use the properties and relations for bi-invariant metrics on compact Lie groups to first construct Riemannian Brownian motion and then proceed to define stochastic differential equations whose invariant distribution is a given density with respect to the Haar measure.

\section{Stochastic analysis in Lie groups}\label{sec:eem}
In this section we recall the Eells-Elworthy-Malliavin construction of Riemannian Brownian motion (RBM) and specialise it to compact Lie groups with bi-invariant metrics. This means that we have to mix the notational conventions of stochastic analysis, Riemannian geometry and Lie theory. Unfortunately, there are many overlaps. We shall use $\boldsymbol{d}$ to denote the It\^o differential and reserve $d$ for the exterior derivative and differentials of maps. We use $g$ to denote a generic group element and, we use $g$ as well as $\langle\,\cdot\,,\,\cdot\,\rangle$ to denote a generic Riemannian metric. No confusion should arise since we never need both in the same context. Indeed, in the context of a compact Lie group,  the bi-invariant metric is always denoted by $Q$. We also have to be cautious with the use of $e$, which in the Riemannian setting denotes a frame $e:\mathbb{R}^d\to T_xM$ at a point $x\in M$, whereas in the Lie group setting it denotes the group identity $e\in G$. We use $\{X_i\}_{i=1}^n$ to denote a basis for the Lie algebra $\mathfrak{g}$.

\begin{definition}
Let $M$ be a smooth manifold and $(\Omega,\mathfrak{F},(\mathfrak{F}_t)_{t\geq 0},\mathbb{P})$ be a filtered probability space. Let $\tau$ be an $\mathfrak{F}$-stopping time. A continuous, $M$-valued process $X$ defind on $[0,\tau)$ is called an $M$-valued semimartingale if $f(X)$ is an $\mathbb{R}$-valued semimartingale on $[0,\tau)$ for all $f\in C^\infty(M)$.
\end{definition}

\subsection{Eells-Elworthy-Malliavin construction}
RBM is the process whose generator is one half the Laplace-Beltrami operator associated with the metric. However, this characterisation does not admit pathwise realisations of the process. A pathwise characterisation of RBM is possible through the Eells-Elworthy-Malliavin construction, as originally shown by \cite{eells1976stochastic, Malliavin1982}. This method is also called rolling without slipping, see \cite{hsu2002stochastic} for a detailed explanation. The construction makes use of the frame bundle $\mathcal{F}(M)$, which is the vector bundle of frames $e:\mathbb{R}^d\to T_xM$ over the $d$-dimensional Riemannian manifold $M$ with the projection $\pi:\mathcal{F}(M)\to M$. A metric induces a torsion-free, metric compatible connection by the fundamental theorem of Riemannian geometry (see \eqref{eq:koszul}). The connection defines a splitting of the frame bundle into a vertical and a horizontal subspace: $\mathcal{F}(M) = \mathcal{V}(M)\oplus\mathcal{H}(M)$. On the frame bundle, the following stochastic differential equation gives rise to RBM
\begin{equation}
    \begin{aligned}
        \boldsymbol{d}U_t = \sum_{i=1}^d H_i(U_t)\circ \boldsymbol{d}W_t^i,
    \end{aligned}
\end{equation}
where $U_t\in\mathcal{F}(M)$, the collection of vector fields $(H_i)_{i=1}^d$ span the horizontal subspace $\mathcal{H}(M)$, and each $W_t^i$ is an $\mathbb{R}$-valued Brownian motion. The initial datum is $U(0)=U_0\in \mathcal{F}(M)$ with $\pi U_0 = x_0\in M$. In local coordinates, this equation takes the form
\begin{equation}
    \begin{aligned}
        \boldsymbol{d}x_t^i &= e^i_j(t)\circ \boldsymbol{d}W_t^j,\\
        \boldsymbol{d}e^i_j(t) &= -\Gamma^i_{k\ell}(x_t)e^\ell_j(t) \circ \boldsymbol{d}x_t^k.
    \end{aligned}
    \label{eq:eemsde}
\end{equation}
with initial data $x(0) = x_0$ and $e(0)= e_0$. We have employed Einstein's convention of summing over repeated indices to keep the notation compact. The first equation in \eqref{eq:eemsde} describes the position process and the second equation describes the parallel transport of the frame along the stochastic path. Hence, on a Riemannian manifold $M$, the position process $x_t=\pi U_t$ is the RBM. 

If the underlying Riemannian manifold is a compact Lie group with a bi-invariant metric $(G,Q)$, we have numerous important simplifications. For bi-invariant metrics, the covariant derivative takes the form \eqref{eq:covariantclg} and it follows that the Christoffel symbols are one-half times the structure constants of the Lie algebra associated with the Lie group. From this it can be deduced immediately that the Christoffel symbols are constant, which implies that the curvature is constant. In addition, \eqref{eq:eemsde} simplifies, since the second equation no longer depends on the position process $X_t$, and we obtain
\begin{equation}
    \begin{aligned}
        \boldsymbol{d}x_t^i &= e_j^i(t)\circ \boldsymbol{d}W_t^j,\\
        \boldsymbol{d}e_j^i(t) &= -\frac{1}{2}C^i_{k\ell}e^\ell_j(t)e^k_m(t)\circ \boldsymbol{d}W_t^m.
    \end{aligned}
\end{equation}
This partial decoupling indicates that the frame can be chosen independently of the position, which is consistent with the fact that the frame bundle over a Lie group is trivial, i.e., the frame bundle over a Lie group admits a global section. Instead of letting the initial point of the RBM on the manifold be arbitrary, we now select the group identity as the initial point. Since a Lie group acts on itself in a transitive manner, one can always find a smooth orbit connecting any point to any other point, so this is without loss of generality. Further, we see that the evolution of the frame is bi-invariant since the structure constants are bi-invariant. Hence, it is no longer necessary to solve the coupled set of equations \eqref{eq:eemsde} since one can fix a frame for all time, taking care to ensure that the frame does correspond to the location on the Lie group. 

The simplest and most natural choice of a frame is the one at the group identity given by a basis $\{X_i\}_{i=1}^n$ of the Lie algebra. This means that the frame together with the Euclidean BM defines a Brownian motion $B_t$ on the Lie algebra
\begin{equation}
    \circ\boldsymbol{d}B_t = X_i\circ \boldsymbol{d}W_t^i \,,
\end{equation}
where $X_i$ is the frame at the group identity. This frame can be transported to a point $x_t\in G$ at time $t$ by a lifted left-translation. This yields the SDE that describes RBM on a compact Lie group with a bi-invariant metric
\begin{equation}
    \boldsymbol{d}g = (dL_{g})_e(X_i\circ\boldsymbol{d}W_t^i) = (dL_g)_e(\circ\boldsymbol{d}B_t) \,,
\end{equation}
which can be solved analytically by means of the Lie exponential for the initial data $g(0) = e\in G$. On compact Lie groups the Lie exponential is surjective (see Theorem \ref{thm:lieexp}) and the solution is given in terms of the time-ordered exponential of the Lie algebra Brownian motion as 
\begin{equation}
    g_t = \mathcal{T}\exp\int_0^t \circ dB_s \,,
\end{equation}
where $\mathcal{T}$ is the time-ordering operator. In the commutative case (e.g. abelian groups), this reduces to the standard exponential of Brownian motion. This shows that RBM on Lie groups with bi-invariant metrics can be obtained without the use of local charts. This is closely related to the injection of a stochastic process on a Lie algebra into a Lie group as introduced by \cite{gangolli1964construction, mckean1969stochastic}. A numerical construction is now particularly straightforward, since one needs only to take $d$ time series corresponding to independent, identically distributed BMs on $\mathbb{R}$, use the basis of the Lie algebra to transform the time series into a BM on the Lie algebra and then apply the matrix exponential map to each matrix in the time series.

In the next section we introduce stochastic differential equations on compact Lie groups with drift.

\subsection{Stochastic differential equations with drift}
In this section we discuss stochastic differential equations on smooth manifolds and specialise to unimodular and compact Lie groups. In \cite{gangolli1964construction, mckean1969stochastic}, solutions to stochastic differential equations were constructed using the underlying Lie group structure. Alternatively, one can also define stochastic differential equations on general smooth manifolds and gradually introduce more algebraic structure. As a result of Malliavin's transfer principle, it is straightforward to define Stratonovich differential equations on a general smooth manifold $M$, see \cite{Malliavin1982, emery2006two}, as a result of Stratonovich integrals satisfying the usual rules of calculus. To define It\^o differential equations, more work is necessary. One method is to embed the manifold into $\mathbb{R}^n$ and specify the It\^o equation as usual. A second method is the intrinsic approach, which requires one to specify both the driving martingale and its quadratic variation. To specify the quadratic variation, second-order geometry is necessary, see \cite{ schwartz1980semi,schwartz1981geometrie,meyer1980differential,meyer1981geometrie}. Such a second-order geometry requires the structure of a connection on the underlying manifold. For both It\^o and Stratonovich integrals on a manifold $M$ with a connection, one can make sense of solutions to stochastic differential equations as $M$-valued semimartingales provided the vector fields are sufficiently smooth. The precise smoothness conditions depend on the situation and which stochastic integral one uses. For the precise details, we refer to \cite{de2024geometric} and here require that the drift be of class $C^{3+\varepsilon}$ and the diffusions of class $C^{4+\varepsilon}$ with $\varepsilon\in(0,1)$.

Following \cite{hsu2002stochastic}, a Stratonovich stochastic differential equation on a manifold is defined by $n$ vector fields $V_\alpha$, $\alpha=1,\hdots,n$ on $M$, meaning that each $V_\alpha(x_t)\in T_{x_t}M$, an $\mathbb{R}^n$-valued driving semimartingale $Z$ and an $M$-valued random variable $x_0\in\mathfrak{F}_0$ that serves as the initial condition. The SDE is written as
\begin{equation}\label{eq:sdeman}
    \boldsymbol{d}x_t = V_\alpha(x_t)\circ\boldsymbol{d}Z_t^\alpha.
\end{equation}
A solution to \eqref{eq:sdeman} is interpreted as follows.
\begin{definition}
    An $M$-valued semimartingale $x$ defined up to a stopping time $\tau$ is a solution of \eqref{eq:sdeman} up to $\tau$ if for all $f\in C^\infty(M)$,
    \begin{equation}
        f(x_t) = f(x_0) + \int_0^t V_\alpha f(x_s)\circ \boldsymbol{d}Z_s^\alpha, \qquad 0\leq t<\tau.
    \end{equation}
\end{definition}

The definition above holds for general smooth manifolds, but in most cases one cannot solve the SDE analytically and one has to resort to numerical methods. On general smooth manifolds one has to supply the charts manually. Selecting suitable charts is a challenging problem in general. The key benefit of Lie groups is this setting is that additional algebraic structure can generate charts in a natural way. The Lie group structure admits us to choose a trivialisation, meaning that we identify $TG\simeq G\times\mathfrak{g}$. The left-trivialisation is defined via the map $G\times\mathfrak{g}\to TG$ defined by $(g,X)\mapsto (g,dL_g(X))$, where $e\in G$ is the identity. Similarly, one can also right-trivialise by using the right-translation. Let $g_t$ satisfy the Stratonovich differential $\boldsymbol{d}g_t = V_\alpha(g_t)\circ \boldsymbol{d}Z_t^\alpha$ with $g(0)=g_0\in \mathfrak{F}_0$. We pull the vector fields $V_\alpha\in T_g G$ back to the identity via the adjoint action and obtain the vectors $X_\alpha\in\mathfrak{g}$ as
\begin{equation}
    X_\alpha(t) = (dL_{g_t^{-1}})_{g_t} V_\alpha(g_t)) \,,
\end{equation}
and then solve
\begin{equation}
    \boldsymbol{d} g_t = (dL_{g_t})_e(X_\alpha)\circ\boldsymbol{d}Z_t^\alpha \,.
\end{equation}
In the next section we give examples of stochastic differential equations on Lie groups. The first set of examples is given by stochastic Lie-Poisson equations and isospectral flows, and the second family of examples consists of Langevin diffusions.

\section{Deterministic and stochastic geometric mechanics}\label{sec:mechanics}
An important example of stochastic differential equations on Lie groups arises in the context of stochastic geometric mechanics. The important equations of geometric mechanics are the Euler-Lagrange equations and Hamilton's equations. When the configuration space is the tangent or cotangent bundle over a Lie group, one can use trivialisation to write the equations on simpler spaces. In the special (though common in mechanics) case where the Lagrangian or Hamiltonian functionals do not depend explicitly on the position variable $g\in G$, symmetry reduction is possible. 

Deterministic Euler-Poincar\'e and Lie-Poisson equations arise through symmetry reduction of Euler-Lagrange equations and Hamilton's equations, \cite{marsden1974reduction, holm1998euler}. These equations describe energy- and geometry-preserving dynamics. The analysis of stochastic Euler-Poincar\'e and Lie-Poisson equations, and identification of suitable numerical methods to solve these equations has been an active area of research for several decades. Without trying to provide an exhaustive history on the subject, we mention \cite{bismut1982mecanique}, which was the first work to consider stochastic Hamiltonian mechanics, with stochastic Lagrangian mechanics not emerging until later. More recently, stochastic mechanics with symmetries and suitable numerical methods for approximating their solution have been considered in \cite{lazaro2008stochastic, bou2009stochastic, deng2014high, holm2015variational, anton2019weak, street2021semi, ST2023, luesink2021casimir, brehier2021splitting, ephrati2024exponential}. By introducing geometric noise via Malliavin's transfer principle into Hamiltonian mechanics and including structure-preserving dissipation, we obtain stochastic differential equations whose invariant measure is also the Gibbs measure.

\subsection{Deterministic geometric mechanics}
We now recall the preliminaries of geometric mechanics on manifolds. We first briefly describe the Lagrangian and Hamiltonian/symplectic picture of classical mechanics, and the relation between the two. We then specialise to mechanics on Lie groups and focus on symplectic mechanics. For more details on symplectic mechanics, including background the differential geometric operations that are used frequently in this section, we refer to \cite{abraham1978foundations}.

At the highest level of generality, these equations describe the dynamics of the momentum map $J:M\to\mathfrak{g}^*$, where $(M,\omega)$ is a $2n$-dimensional symplectic manifold acted upon transitively and symplectomorphically by a (finite-dimensional) Lie group $G$, and $\mathfrak{g}^*$ is the dual of the Lie algebra corresponding to $G$, see \cite{souriau1970structure}. Here $\omega$ is the symplectic form, which is a closed and nondegenerate 2-form. A transitive action means that for every two points on the manifold there is a group element such that the action maps one point onto the other. A symplectomorphic action means that the group action preserves the symplectic form. 

In many cases, such as for rigid body dynamics, $M=T^*G$, i.e., the symplectic manifold is the cotangent bundle over a Lie group, which means that the symplectic form $\omega$ is the exterior derivative of the Liouville 1-form: $\omega=-d\theta$. By Darboux's theorem, one always has a set of canonical local coordinates on a symplectic manifold, in particular, we write local coordinates $(g,p)$ for $T^*G$. In these coordinates, the Liouville 1-form is $\theta =p\,dg$ and the (canonical) symplectic form is $\omega = dg\wedge dp$. Since the left action $L_g:G\to G$ is a diffeomorphism, $(dL_g)_e:T_eG\simeq\mathfrak{g}\to T_gG$. Via the inverse at $g\in G$, $(dL_{g^{-1}})_g:T_gG\to \mathfrak{g}$, we can trivialise $TG$ via $(g,\dot{g})\mapsto (g,(dL_{g^{-1}})_g\dot{g}) \equiv (g,X)$. This map enjoys several different equivalent notations, $(dL_g)_e = T_eL_g = (L_g)_*$. Similarly, since the right action $R_{g^{-1}}:G\to G$ is also a diffeomorphism and $(dR_{g^{-1}})_e:T_eG\simeq\mathfrak{g}\to T_gG$, it too induces a trivialisation of $TG$ via $(g,\dot{g})\mapsto(g, (dR_g)_g\dot{g})$, where one can identify $(dR_g)_g\dot{g}$ with an element in the Lie algebra. In a similar manner, making use of the cotangent lifted group operation, we can trivialise $T^*G\simeq G\times\mathfrak{g}^*$ via $(g,p)\mapsto (g,(dL_g)^*_e p) \equiv (g,m)$. Here $m=(dL_g)^*_e p\in\mathfrak{g}^*$ is the left-trivialised momentum. 

To obtain equations of motion, one considers the Euler-Lagrange equations associated with a Lagrangian $\mathcal{L}:TG\to\mathbb{R}$. By means of the trivialisation, we obtain the Lagrangian $\mathcal{L}:G\times\mathfrak{g}\to\mathbb{R}$. Let $\frac{\delta\mathcal{L}}{\delta g}$ and $\frac{\delta\mathcal{L}}{\delta X}$ denote the variational derivatives (see \cite{gelfand2000calculus} for detailed definitions) of the Lagrangian $\mathcal{L}$ with respect to $g\in G$ and $X\in\mathfrak{g}$, respectively. The left-trivialised Euler-Lagrange equations take the form
\begin{equation}\label{eq:eulerlagrange}
    \begin{aligned}
            \frac{\boldsymbol{d}}{\boldsymbol{d}t}\frac{\delta\mathcal{L}}{\delta X} &= \mathrm{ad}^*_X \frac{\delta\mathcal{L}}{\delta X} + (dL_{g^{-1}})^*_g\frac{\delta\mathcal{L}}{\delta g},\\
            \frac{\boldsymbol{d}}{\boldsymbol{d}t}g &= (dL_g)_eX.
    \end{aligned}
\end{equation}
These equations are in fact those appearing in \cite{poincare1901forme}, in which it is noted that should the Lagrangian lose its group dependence during the trivialisation procedure, then equations \eqref{eq:eulerlagrange} are particularly special and are known as Euler-Poincar\'e equations. For a modern discussion of these equations and their detailed derivation, see \cite{marle2013henri}. The Euler-Lagrange equations can be transformed to Hamilton's canonical equations via the Legendre transform. The Legendre transform defines the Hamiltonian $\mathcal{H}:T^*G\simeq G\times\mathfrak{g}^*\to\mathbb{R}$ in terms of the Lagrangian or vice versa provided that the fiber derivatives are diffeomorphisms, this condition is also known as hyperregularity of the Lagrangian and Hamiltonian. Let $\langle\,\cdot\,,\,\cdot\,\rangle:T^*G\times TG\to\mathbb{R}$ denote a natural duality pairing (usually the duality pairing is taken to be the metric, under which one may identify the fibers of $T^*G$ with the fibers of $TG$). Provided that the Lagrangian and/or Hamiltonian is hyperregular, the Legendre transform is a diffeomorphism. The Legendre transform is defined as $\mathcal{H}(g,m) = \langle m,X\rangle - \mathcal{L}(g,X)$, which results in identifying $m=\frac{\delta\mathcal{L}}{\delta X}\in\mathfrak{g}^*$, $X=\frac{\delta\mathcal{H}}{\delta m}\in\mathfrak{g}$ and $\frac{\delta\mathcal{L}}{\delta g} = - \frac{\delta\mathcal{H}}{\delta g}\in T^*_gG$. Substituting these identities into the equations above, one obtains the trivialised form of Hamilton's canonical equations
\begin{equation}\label{eq:hamcan}
    \begin{aligned}
            \frac{\boldsymbol{d}}{\boldsymbol{d}t}m &= \mathrm{ad}^*_{\delta\mathcal{H}/\delta m} m - (dL_{g^{-1}})^*_g\frac{\delta\mathcal{H}}{\delta g},\\
            \frac{\boldsymbol{d}}{\boldsymbol{d}t}g &= (dL_g)_e\frac{\delta\mathcal{H}}{\delta m}.
    \end{aligned}
\end{equation}
Since the tangent and cotangent bundles are finite-dimensional, variational derivatives can be interpreted as partial derivatives with respect to local coordinates. In the Hamiltonian form \eqref{eq:hamcan}, if the underlying group is abelian, then $\mathrm{ad}^*$ is trivial and the differential of left-translation is the identity, and one obtains the equations of Hamiltonian mechanics in their canonical form.

One can also derive \eqref{eq:hamcan} without the use of the calculus of variations. In the Hamiltonian approach, one uses the symplectic form on $M=T^*G$ to define Hamiltonian vector fields. Let $\mathcal{H}:G\times\mathfrak{g}^*\to\mathbb{R}$ be the trivialised Hamiltonian. On symplectic manifolds, one defines Hamiltonian vector fields $X_\mathcal{H}$ as vector fields that satisfy the equation $i_{X_\mathcal{H}}\omega = d\mathcal{H}$, where the left-hand side denotes the interior product between the vector field $X_\mathcal{H}$ and the symplectic form $\omega$. As a result of the cotangent bundle structure and the existence of the Liouville 1-form $\theta$, we can define Hamiltonian vector fields as vector fields that satisfy $-i_{X_\mathcal{H}}\theta = \mathcal{H}$, since the differential of this relation implies the symplectic one. Letting the group of time-translations act symplectomorphically on the relation $i_{X_\mathcal{H}}\omega = d\mathcal{H}$ also implies \eqref{eq:hamcan}. An important feature of the Hamiltonian formulation is the fact that the Hamiltonian appears only on the right-hand side of \eqref{eq:hamcan}. Let $S=(m,g)$, then we can express \eqref{eq:hamcan} using the Hamiltonian vector field $X_\mathcal{H}$ as
\begin{equation}\label{eq:hamcanstate}
    \frac{\boldsymbol{d}}{\boldsymbol{d}t}S = X_\mathcal{H}(S).
\end{equation}
In this way, one can read off the coordinate expression of $X_\mathcal{H}$ from \eqref{eq:hamcan}. Further, one can define the Poisson bracket $\{\,\cdot\,,\,\cdot\,\}:C^\infty(M)\times C^\infty(M)\to C^\infty(M)$ as
\begin{equation}\label{eq:poissonbracket}
\begin{aligned}
    \{\mathcal{F},\mathcal{G}\} &= X_\mathcal{F}(\mathcal{G}) = -X_\mathcal{G}(F) = \omega(X_\mathcal{F},X_\mathcal{G})\\
    &= \left\langle m,\left[\frac{\delta\mathcal{F}}{\delta m},\frac{\delta\mathcal{G}}{\delta m}\right]\right\rangle + \left\langle\frac{\delta\mathcal{G}}{\delta m},(dL_{g^{-1}})_g\frac{\delta\mathcal{F}}{\delta g}\right\rangle - \left\langle\frac{\delta\mathcal{F}}{\delta m}, (dL_{g^{-1}})_g\frac{\delta\mathcal{G}}{\delta g}\right\rangle 
\end{aligned}
\end{equation}
where $X_\mathcal{F}$ and $X_\mathcal{G}$ are the Hamiltonian vector fields corresponding to the functions $\mathcal{F},\mathcal{G}\in C^\infty(M)$ and the second line is the left-trivialised Poisson bracket on $M=T^*G\simeq G\times\mathfrak{g}^*$ expressed in Darboux coordinates. The pairing is a natural pairing between $\mathfrak{g}$ and its dual, which in the reductive case is the bi-invariant metric. On a symplectic manifold, one therefore has a natural algebra on smooth functions over the manifold, known as the Poisson algebra. This is an infinite-dimensional Lie algebra, because the Poisson bracket above has the properties of an abstract Lie bracket. It further satisfies the Leibniz rule, meaning that a Poisson algebra is also a derivation algebra. This means that one can write the equations of symplectic mechanics through observables $\mathcal{F}: C^\infty(M)\times\mathbb{R}_+\to \mathbb{R}$ (time-dependent functions on the symplectic manifold) as
\begin{equation}
    \frac{\boldsymbol{d}}{\boldsymbol{d} t}\mathcal{F} = \frac{\partial}{\partial t}\mathcal{F} - \{\mathcal{F},\mathcal{H}\},
\end{equation}
and talk about invariant observables whenever $\frac{\boldsymbol{d}}{\boldsymbol{d}t}\mathcal{F} = 0$. Under the $L^2$-pairing on a manifold without boundary (or with appropriate boundary conditions) with respect to the symplectic volume measure, functions are dual to densities as the boundary terms vanish by Stokes' theorem or the boundary conditions. On symplectic manifolds, we further have Liouville's theorem, which states that the symplectic volume form $\omega^n$ is invariant under the flow of a Hamiltonian system. Let $\rho$ be any invariant density with respect to the symplectic volume measure, then it can be shown that $\rho$ must satisfy
\begin{equation}\label{eq:liouville}
    \frac{\boldsymbol{d}}{\boldsymbol{d} t}\rho = \frac{\partial}{\partial t}\rho - \{\mathcal{H},\rho\} = 0,
\end{equation}
Equation \eqref{eq:liouville} is the Liouville--Von Neumann equation and it describes the evolution of the time-dependent distribution function $\rho(t,g,m)$ under symplectic dynamics. Taking $\rho=\rho(\mathcal{H})$, it follows that the invariant measure of a deterministic symplectic mechanical system is the Gibbs measure $Z^{-1}e^{-\beta\mathcal{H}}|\omega^n|$, where $Z$ is an appropriate normalisation. 

In the next section, we employ Malliavin's transfer principle to obtain the equations of stochastic geometric mechanics with structure-preserving noise.

\subsection{Stochastic geometric mechanics}\label{sec:stochgeomech}
By means of Malliavin's transfer principle, it is straightforward to lift deterministic geometric mechanics to stochastic geometric mechanics. Let $(M,\omega)$ be a $2n$-dimensional symplectic manifold. Given a collection of Hamiltonians $\mathcal{H}_\alpha:C^\infty(M)\to\mathbb{R}$ for $\alpha=0,1,\hdots, n$ and a family of $\mathbb{R}$-semimartingales $(Z^\alpha)_{\alpha=0}^n$ with the convention $Z^0 = t$, one can formulate stochastic mechanics as
\begin{equation}\label{eq:stochhamcan}
    \boldsymbol{d}S = X_{\mathcal{H}_\alpha}(S)\circ \boldsymbol{d}Z_t^\alpha,
\end{equation}
where $S_t:\mathbb{R}\to M$ describes the state of the system. Equation \eqref{eq:stochhamcan} is the structure-preserving stochastic version of deterministic symplectic Hamiltonian mechanics \eqref{eq:hamcanstate} if one takes the drift Hamiltonian $\mathcal{H}_0$ in \eqref{eq:stochhamcan} to be the same as the Hamiltonian for the deterministic system. In Darboux coordinates with $M\simeq G\times\mathfrak{g}^*$, one obtains
\begin{equation}\label{eq:stochsymplectic}
    \begin{aligned}
        \boldsymbol{d}m &= \left(\mathrm{ad}^*_{\delta\mathcal{H}_\alpha/\delta m}m -(dL_{g^{-1}})_g\frac{\delta\mathcal{H}_\alpha}{\delta g}\right)\circ \boldsymbol{d}Z_t^\alpha,\\
        \boldsymbol{d}g&= (dL_g)_e\frac{\delta\mathcal{H}_\alpha}{\delta m}\circ \boldsymbol{d}Z_t^\alpha.
    \end{aligned}
\end{equation}

\begin{remark}
    Wellposedness of the deterministic equations \eqref{eq:hamcan} and the stochastic equations \eqref{eq:stochsymplectic} of deterministic and stochastic symplectic mechanics depend crucially on the Hamiltonians and on compactness properties of the domains. On compact domains, Lipschitz continuity of the vector fields is enough for global wellposedness. On noncompact domains, this is only enough for local wellposedness.
\end{remark}

In many physical applications of the equations of geometric mechanics, symmetry reduction is possible as shown by \cite{marsden1974reduction} and \cite{holm1998euler}. The necessary requirement for symmetry reduction to be possible is that the Lagrangian and/or Hamiltonian is left- or right-invariant under the action of a Lie group. This in fact guides whether one should left- or right-trivialise the equations, since, informally speaking, such invariance implies that the terms involving variational derivatives with respect to group elements drop out of the equations. As a consequence, the equations can be formulated solely on the dual of the Lie algebra. Malliavin's transfer principle does not hinder this symmetry reduction, as shown in \cite{de2020implications, ST2023}. In the symplectic interpretation, symmetry reduction of Hamilton's canonical equations leads to the Lie-Poisson equation
\begin{equation}\label{eq:lie-poissonsde}
     \begin{aligned}
        \boldsymbol{d} m &= \mathrm{ad}^*_{\delta\mathcal{H}_\alpha/\delta m} m \circ \boldsymbol{d}Z_t^\alpha,
    \end{aligned}
\end{equation}
where the equation describing the evolution of $g$ has become redundant. This is an important benefit since this means that the differential equations are formulated completely on vector spaces rather than on vector bundles over manifolds. By keeping the evolution of $g$, one can completely reconstruct the dynamics on the cotangent bundle of a Lie group from dynamics on the dual of the Lie algebra, which is of course a significant reduction in dimension.

The equations in this section require that the phase space is the cotangent bundle of a Lie group, but no assumptions are made on the type of Lie group. In the next section, we show that if the Lie group is reductive, then we obtain a special class of mechanical systems with the property that the eigenvalues of the initial condition are preserved along trajectories of the system.

\subsection{Stochastic isospectral flows} 
A related family of differential equations to the Lie-Poisson equations is the family of isospectral flows. Via Malliavin's transfer principle, one can introduce noise into such equations without altering their geometric structure. Let $\{Z^\alpha\}^n_{\alpha= 0}$ denote the collection of semimartingales as before. A stochastic isospectral flow is described by a differential equation of the form
\begin{equation}\label{eq:ispsectralsde}
    \boldsymbol{d}Y_t = [B_\alpha(Y_t),Y_t]\circ \boldsymbol{d}Z_t^\alpha \,,
\end{equation}
where $W_t\in S\subset\mathfrak{gl}(n,\mathbb{C})$ and each $B_\alpha:S\to\mathfrak{n}(S)$. Here, $\mathfrak{n}(S)$ denotes the normaliser algebra of the linear subspace $S$ of $\mathfrak{gl}(n,\mathbb{C})$, which is defined as follows: given a Lie group $G$ with Lie algebra $\mathfrak{g}$, let $S\subset G$, the normaliser algebra of $S$ is the set $\mathfrak{n}(S) = \{\xi\in \mathfrak{g} \,|\,[\xi,S]\subseteq S\}$ and is a Lie subalgebra of $\mathfrak{g}$. 

Common examples of such flows are given for $S=\mathrm{Sym}(n,\mathbb{R})$, the space of symmetric real $n\times n$ matrices, for which the normaliser algebra is the space of skew-symmetric real matrices $\mathfrak{n}(S)=\mathfrak{so}(n)$. Another important example is when $S=\mathfrak{g}\subset\mathfrak{gl}(n,\mathbb{C})$ is a Lie subalgebra, for which the normaliser $\mathfrak{n}(\mathfrak{g})=\mathfrak{g}$. In this case, one may ask under which conditions an isospectral flow is a Lie-Poisson equation and vice-versa. This question is answered in \cite{modin2020lie}, in which the isospectral symplectic Runge-Kutta methods are developed. This is a wonderful class of numerical methods for solving differential equations on Lie groups, because it applies precisely in the setting where the Lie group is reductive. Let us recall the argument that shows this next.

First, recall that $\langle\mathrm{ad}^*_X m,Y\rangle = -\langle m,[X, Y]\rangle$. Without loss of generality, one may assume that $\mathfrak{g}\subset\mathfrak{gl}(n,\mathbb{C})$. We can identify $\mathfrak{gl}(n,\mathbb{C})^*$ with $\mathfrak{gl}(n,\mathbb{C})$ via the Frobenius pairing $\langle Y,X\rangle = \mathrm{tr}(Y^\dagger X)$. Similarly, one may identify $\mathfrak{g}^*$ with the subspace $\mathfrak{g}\subset\mathfrak{gl}(n,\mathbb{C})$. Upon extending the Hamiltonians $\mathcal{H}_\alpha:\mathfrak{g}^*\to\mathbb{R}$ to all of $\mathfrak{gl}(n,\mathbb{C})$ by defining each Hamiltonian to be constant on the affine spaces given by translations of the orthogonal complement of $\mathfrak{g}$, one can identify derivatives of the Hamiltonians with the gradient and write $\nabla \mathcal{H}_\alpha$, where the gradient is defined with respect to the Frobenius inner product as follows
\begin{equation}
    \langle\nabla \mathcal{H}_\alpha(Y), X\rangle := \frac{d}{d\varepsilon}\Big|_{\varepsilon=0}\mathcal{H}_\alpha(Y\exp\varepsilon X).
\end{equation}
We can now write the Lie-Poisson system \eqref{eq:lie-poissonsde} in almost isospectral form as
\begin{equation}
    \boldsymbol{d}Y_t = -\Pi[\nabla \mathcal{H}_\alpha(Y_t)^\dagger,Y_t]\circ\boldsymbol{d}Z_t^\alpha
\end{equation}
where $\Pi:\mathfrak{gl}(n,\mathbb{C})\to\mathfrak{g}$ is the orthogonal projection. This holds for any finite-dimensional Lie algebra. If the Lie algebra $\mathfrak{g}$ is reductive, the adjoint representation is closed under hermitian transpose, i.e. $[X,Y]^\dagger = [Y,X]$. In this case, the Lie-Poisson equation becomes a true isospectral flow
\begin{equation}
    \boldsymbol{d}Y_t = [\nabla \mathcal{H}_\alpha(Y_t)^\dagger, Y_t]\circ \boldsymbol{d}Z_t^\alpha.
\end{equation}
Conversely, an isospectral flow is Lie-Poisson whenever $B_\alpha(Y)=\nabla \mathcal{H}_\alpha(Y)^\dagger$, since it can always be extended to a Lie-Poisson system on $\mathfrak{gl}(n,\mathbb{C})$ or possibly a smaller reductive algebra that contains $S$. This extension is required because $S$ is not necessarily the dual of a Lie algebra. By defining the Hamiltonian to be constant on the orthogonal complement of $S$, the system naturally foliates into invariant affine subspaces generated by $S$. The Toda lattice is an important example of a deterministic isospectral flow that can be made Lie-Poisson in this way, see \cite{modin2020lie}. Numerical integrators for stochastic isospectral Lie-Poisson equations are discussed in \cite{ephrati2024exponential} and have the benefit that they do not rely on algebra-to-group maps, which are numerically expensive to evaluate in high dimensions. 

If the drift Hamiltonian is given by a bi-invariant metric, the gradient of the drift Hamiltonian aligns with $Y_t$ and the drift part of the isospectral flow vanishes. This means that there are no nontrivial isospectral deformations of a bi-invariant metric. This fact helps us in the next section, where it simplifies the kinetic Langevin equations used for sampling from Gibbs measures.

\section{Sampling from Gibbs measures}\label{sec:sampling}
In this section we discuss approaches to sampling from Gibbs measures on symplectic manifolds and Lie groups. In the previous section we recalled the key fact that deterministic symplectic mechanics have the Gibbs measure as their invariant measure. We then constructed stochastic symplectic mechanics by means of Malliavin's transfer principle. In this section, we first show that the Gibbs measure satisfies a maximum entropy principle on symplectic manifolds and proceed to study the invariant measure of stochastic symplectic mechanics. We show that it cannot be the Gibbs measure and by introducing double-bracket dissipation, we show that stochastic symplectic mechanics with dissipation has the Gibbs measure again as its invariant measure.

Let $\mathcal{H}_0:G\times\mathfrak{g}^*\to\mathbb{R}$ be a Hamiltonian, then the Gibbs measure $\mathbb{P}_\infty$ is the probability measure defined as
\begin{equation}
    \mathbb{P}_\infty = Z^{-1}e^{-\beta \mathcal{H}_0}\mu(dg)\lambda(dm), \qquad Z = \int_{G\times\mathfrak{g}^*} e^{-\beta \mathcal{H}_0}\mu(dg)\lambda(dm),
\end{equation}
where the inverse temperature $\beta >0$ is a free parameter. The Gibbs measure $\mathbb{P}_\infty$ follows equivalently as the maximal entropy probability measure among all configurations with fixed average energy, as given by the following maximum entropy principle that holds for any symplectic manifold.

\begin{proposition}\label{prop:gibbsvar}
Let $(M,\omega)$ be a $2n$-dimensional symplectic manifold and let $|\omega^n|$ denote the completion of the volume form $\omega^n$ to a Borel measure. The Gibbs measure is a solution to the constrained variational principle
    \begin{equation}
        \mathbb{P}_\infty = \underset{\mu \gg \lambda}{\mathrm{arg}\max}\{-\mathcal{H}(\mu|\lambda)\} \,,\quad \text{ such that}\quad  \int_{M} \mathcal{H}_0 \,d\mu =c \quad\text{and}\quad  \int_{M} d\mu = 1 \,,
    \end{equation}
    where $c\in\mathbb{R}$, $\mathcal{H}(\mu|\lambda) := \int_{M} \frac{d\mu}{d\lambda}\log \frac{d\mu}{d\lambda}d\lambda$ is the relative entropy (or Kullback-Leibler divergence) of the measure $\mu$ with respect to a reference measure $\lambda$, and $\mathbb{P}_\infty = Z^{-1}e^{-\beta\mathcal{H}_0}|\omega^n|$ with $Z=\int_M e^{-\beta\mathcal{H}_0}|\omega^n|$. 
\end{proposition}
Here $d\mu/d\lambda$ denotes the Radon-Nikodym derivative of $\mu$ with respect to $\lambda$. The first constraint sets the average energy level of the system and the second constraint enforces that $\mu$ is a probability measure. A natural choice for the reference measure on $G\times\mathfrak{g}^*$ is the Liouville measure $|\omega^n|=\mu(dg)\lambda(dm)$.
\begin{proof}
    The proof of this follows the argument given in \cite{ST2023}. Given that $\mu\gg \lambda$, without loss of generality, we can set $\mu = \phi\,\lambda$ for some positive function $\phi:M\to\mathbb{R}$. By introducing Lagrange multipliers $\beta,\gamma$ to enforce the constraints, we formulate the constrained optimisation problem as a minimisation problem of the following functional
    \begin{equation}
        S[\phi,\beta,\gamma] = \int_M\phi\log\phi \,d\lambda + \beta\left(\int_M \mathcal{H}_0 \phi\,d\lambda - c\right) + \gamma\left(\int_M \phi\, d\lambda - 1\right).
    \end{equation}
    Requiring the first variation of $S$ to vanish implies the equations
    \begin{equation}
    \begin{aligned}
        \log\phi+1+\beta \mathcal{H}_0 + \gamma = 0,\\
        \int_M \mathcal{H}_0\phi\,d\lambda = c,\\
        \int_M\phi\,d\lambda = 1.
    \end{aligned}
    \end{equation}
    From the first equation we obtain
    \begin{equation}
        \phi = Z^{-1}e^{-\beta \mathcal{H}_0}
    \end{equation}
    with $Z=e^{1+\gamma}$ and the third equation then implies that $Z=\int_M e^{-\beta \mathcal{H}_0}d\lambda$. Let $\mu_*$ be the measure defined by the above choice of $\phi$. We now show that $\mu_*$ is the maximal entropy measure. For any probability measure $\mu\gg\lambda$ satisfying $\int_M \mathcal{H}_0d\mu = c$, we have
    \begin{equation}
        \begin{aligned}
            -\mathcal{H}(\mu_*|\lambda)+\mathcal{H}(\mu|\lambda) &= -\int_M\frac{d\mu_*}{d\lambda}\log\frac{d\mu_*}{d\lambda} d\lambda + \int_M \frac{d\mu}{d\lambda}\log \frac{d\mu}{d\lambda}d\lambda \\
            &= \int_M (\log Z + \beta \mathcal{H}_0)d\mu_* + \int_M\log \frac{d\mu}{d\lambda}d\mu\\
            &= \log Z +\beta c + \int_M\log \frac{d\mu}{d\lambda}d\mu\\
            &= \int_M(\log Z + \beta \mathcal{H}_0)d\mu + \int_M\log\frac{d\mu}{d\lambda}d\mu\\
            &= -\int_M \log \frac{d\mu_*}{d\lambda} d\mu + \int_M\log \frac{d\mu}{d\lambda}d\mu\\
            &= \int_M\log\frac{d\mu_*}{d\mu} d\mu\\
            &= \mathcal{H}(\mu|\mu_*).
        \end{aligned}
    \end{equation}
    Since the relative entropy is always positive, we have shown that $-\mathcal{H}(\mu_*|\lambda)+\mathcal{H}(\mu|\lambda) = \mathcal{H}(\mu|\mu_*)>0$, which proves that $\mu_*$ is the maximal entropy measure.
\end{proof}

\subsection{Invariant measure for stochastic symplectic dynamics}
Deterministic Hamiltonian systems on symplectic manifolds have Gibbs measures as their invariant measure, which motivates the use of Hamiltonian Monte-Carlo methods. In these methods, one implements a numerical method that preserves the symplectic structure, such as the St\"ormer-Verlet method or implicit midpoint rule, followed by a Metropolis-Hastings step. Since symplectic Hamiltonian systems are time-reversible, distant proposals tend to be accepted with high probability, which in turn decreases the correlation between samples. However, as shown in \cite{duncan2016variance}, ergodic irreversible diffusions tend to converge faster to their target distributions. To this end, several irreversible approaches based on Langevin dynamics have been proposed. In \cite{arnaudon2019irreversible} a Langevin-type Markov chain Monte-Carlo algorithm is proposed for sampling from Gibbs measures on Lie groups motivated by a stochastic version of symplectic mechanics. 

Through Malliavin's transfer principle, in Section \ref{sec:stochgeomech} we obtained a stochastic version of Hamilton's equations on the cotangent bundle of a Lie group that preserves the underlying geometric structure. 
In case the family of $\mathbb{R}$-semimartingales $(Z^\alpha)^n _{\alpha=0} = (t,W_t^1,W_t^2,\hdots, W_t^n)$, the It\^o form of \eqref{eq:stochhamcan} is given by
\begin{equation}\label{eq:stochhamcanito}
    \boldsymbol{d}S_t = \left(X_{\mathcal{H}_0}(S_t)+ \frac{1}{2}X_{\mathcal{H}_i}X_{\mathcal{H}_i}(S_t)\right)\boldsymbol{d}t + X_{\mathcal{H}_i}(S_t)\boldsymbol{d}W_t^i.
\end{equation}
In the It\^o interpretation as above, we can compute the generator of the process $S_t$. We denote the generator by $\mathcal{L}$, which should not be confused with the Lagrangian. By definition, the generator of the process is given by 
\begin{equation}
    \mathcal{L}\mathcal{F}(S) = \lim_{\epsilon\to 0}\frac{\mathbb{E}[\mathcal{F}(S_\epsilon)\,|\, S_0 = S] - \mathcal{F}(S)}{\epsilon}.
\end{equation}
The expectation eliminates the martingale term in \eqref{eq:stochhamcanito} and using the relations between Hamiltonian vector fields and Poisson brackets in \eqref{eq:poissonbracket}, it follows that the generator (or forward Kolmogorov operator) is given by
\begin{equation}
    \mathcal{L}\mathcal{F} = \{\mathcal{H}_0,\mathcal{F}\} + \frac{1}{2}\{\mathcal{H}_i,\{\mathcal{H}_i,\mathcal{F}\}\}.
\end{equation}
Note that in absence of noise, i.e., when the diffusion Hamiltonians $\mathcal{H}_i$ are zero, the generator is self-adjoint with respect to the $L^2$-pairing with the measure given by the symplectic volume form. Hence, it follows that the Fokker-Planck equation, which governs the probability density function $\rho$, involves the same operator. This in turn means that we obtain the right-hand side of the Liouville equation \eqref{eq:liouville}, which is known to have the Gibbs measure as its equilibrium solution. It follows that in general, Malliavin's transfer principle, while preserving the geometric structure, does not preserve the invariant measure of the deterministic system. This can be remedied by introducing a special type of dissipation into the system. This type of dissipation is known as double-bracket dissipation (see \cite{brockett1991dynamical}, \cite{bloch1996euler} and references therein for details) and the Stratonovich SDE \eqref{eq:stochhamcan} with $(Z^\alpha)^n _{\alpha=0} = (t,W_t^1,W_t^2,\hdots, W_t^n)$ changes as follows
\begin{equation}\label{eq:dbdstandard}
    \boldsymbol{d}S_t = \left(X_{\mathcal{H}_0}(S_t) + \frac{\beta}{2}\{\mathcal{H}_0,\mathcal{H}_i\}X_{\mathcal{H}_i}(S_t)\right)\boldsymbol{d}t + X_{\mathcal{H}_i}(S_t)\circ\boldsymbol{d}W_t^i
\end{equation}
The It\^o form is then given by
\begin{equation}\label{eq:dbdito}
    \boldsymbol{d}S_t = \left(X_{\mathcal{H}_0}(S_t) + \frac{\beta}{2}\{\mathcal{H}_0,\mathcal{H}_i\}X_{\mathcal{H}_i}(S_t) + \frac{1}{2}X_{\mathcal{H}_i}X_{\mathcal{H}_i}(S_t)\right)\boldsymbol{d}t + X_{\mathcal{H}_i}(S_t)\boldsymbol{d}W_t^i.
\end{equation}
Balancing the effect of Hamiltonian noise with double-bracket dissipation was considered in \cite{arnaudon2018noise} in the setting of coadjoint orbits and in \cite{arnaudon2019irreversible} in the setting of Lie groups.
\begin{proposition}\label{prop:langevin}
    Equation \eqref{eq:dbdito} has the Gibbs measure $\mathbb{P}_\infty = e^{-\beta\mathcal{H}_0}|\omega^n|$ as its invariant measure on a symplectic manifold without boundary.
\end{proposition}
\begin{proof}
    The generator of \eqref{eq:dbdito} can be directly read off and is given by $\mathcal{L}\mathcal{F} = \{\mathcal{H}_0,\mathcal{F}\} + \frac{\beta}{2}\{\mathcal{H}_0,\mathcal{H}_i\}\{\mathcal{H}_i,\mathcal{F}\} + \frac{1}{2}\{\mathcal{H}_i,\{\mathcal{H}_i,\mathcal{F}\}\}$. We can compute the adjoint of the generator by repeatedly using the Leibniz rule and the skew-symmetry of the Poisson bracket
    \begin{equation}
    \begin{aligned}
        \int_M \mathcal{F}\mathcal{L}^*\rho |\omega^n| &= \int_M \rho \mathcal{L}\mathcal{F}|\omega^n|\\
        &= \int_M\left(\rho\{\mathcal{H}_0,\mathcal{F}\} + \frac{\beta}{2}\rho\{\mathcal{H}_0,\mathcal{H}_i\}\{\mathcal{H}_i,\mathcal{F}\} - \frac{1}{2}\rho\{\mathcal{H}_i,\{\mathcal{H}_i,\mathcal{F}\}\}\right) |\omega^n| \\
        &= \int_M \left(\{\mathcal{H}_0, \rho \mathcal{F}\} + \frac{\beta}{2}\{\mathcal{H}_i,\rho \mathcal{F}\{\mathcal{H}_0,\mathcal{H}_i\}\} - \frac{1}{2}\{\mathcal{H}_i,\rho\{\mathcal{H}_i,\mathcal{F}\}\}+\frac{1}{2}\{\mathcal{H}_i,\mathcal{F}\{\mathcal{H}_i,\rho\}\}\right)|\omega^n|\\
        &\quad + \int_M \left(\mathcal{F}\{\rho,\mathcal{H}_0\} +\frac{\beta}{2}\mathcal{F}\{\rho\{\mathcal{H}_0,\mathcal{H}_i\},\mathcal{H}_i\} + \frac{1}{2}\mathcal{F}\{\{\rho,\mathcal{H}_i\},\mathcal{H}_i\}\right)|\omega^n|
    \end{aligned}
    \end{equation}
    The manifold by assumption has empty boundary, so integration by parts does not give nontrivial boundary terms. By definition, the Lie derivative of the symplectic form along a Hamiltonian vector field vanishes, i.e., for $\mathcal{F}\in C^\infty(M)$, $\pounds_{X_\mathcal{F}}\omega = 0$, which implies Liouville's theorem: $\pounds_{X_\mathcal{F}}\omega^n = 0$. Since the Poisson bracket of any two functions $\mathcal{F},\mathcal{G}\in C^\infty(M)$ is itself a smooth function on $M$, it follows that $\pounds_{\{\mathcal{F},\mathcal{G}\}}\omega = 0$. Hence the interior product of the Poisson bracket with the symplectic volume form vanishes: it holds that $\int_M \{\mathcal{F},\mathcal{G}\}\omega^n = 0$. Therefore, the adjoint of the generator is
    \begin{equation}
        \mathcal{L}^*\rho = \{\rho,\mathcal{H}_0\} +\frac{\beta}{2}\{\rho\{\mathcal{H}_0,\mathcal{H}_i\},\mathcal{H}_i\} + \frac{1}{2}\{\{\rho,\mathcal{H}_i\},\mathcal{H}_i\}.
    \end{equation}
    Taking $\rho_\infty = e^{-\beta\mathcal{H}_0}$, we compute the Fokker-Planck equation
    \begin{equation}
    \begin{aligned}
        \frac{\partial}{\partial t}\rho_\infty &= \mathcal{L}^*\rho_\infty = \{\rho_\infty,\mathcal{H}_0\}+\frac{\beta}{2}\{\rho_\infty\{\mathcal{H}_0,\mathcal{H}_i\},\mathcal{H}_i\} + \frac{1}{2}\{\{\rho_\infty,\mathcal{H}_i\},\mathcal{H}_i\}\\
        &= -\beta e^{-\beta\mathcal{H}_0}\{\mathcal{H}_0,\mathcal{H}_0\}+\frac{\beta}{2}\{e^{-\beta\mathcal{H}_0}\{\mathcal{H}_0,\mathcal{H}_i\},\mathcal{H}_i\} - \frac{\beta}{2}\{e^{-\beta\mathcal{H}_0}\{\mathcal{H}_0,\mathcal{H}_i\},\mathcal{H}_i\}\\
        &= 0.
    \end{aligned}
    \end{equation}
\end{proof}
The above proof shows that double-bracket dissipation term precisely cancels the It\^o correction for the Gibbs measure and that this works no matter what one chooses for the Hamiltonians. Indeed, in case there is no noise, the double bracket term vanishes and if the diffusion Hamiltonians commute with the drift Hamiltonian, then the double bracket term also vanishes. Hence, the above result can be viewed as a type of fluctuation-dissipation theorem. Note that the inverse temperature $\beta$ is not a free parameter, as it explicitly appears in the double-bracket term.

\begin{remark}
    The above proof using abstract Poisson brackets can also be performed directly, though it should be noted that for $T^*G\simeq G\times\mathfrak{g}^*$ a single Poisson bracket contains three terms (see \eqref{eq:poissonbracket}). The It\^o-Stratonovich correction then in general consists of nine terms, making the computation rather unwieldy.
\end{remark}

We now consider several example applications of this stochastic symplectic framework. All examples have the Gibbs measure as their invariant measure since they are all special cases of \eqref{eq:dbdstandard} and Proposition \ref{prop:langevin}. In all of the cases below, we take the diffusion Hamiltonians to be linear in the Darboux coordinates, but more general choices are of course possible.

\subsection{Examples}
\begin{example}[Momentum Langevin on reductive Lie groups]
    Let $G$ be a compact semisimple Lie group with bi-invariant metric $Q$. Choose the semimartingale terms as $(Z^\alpha)_{\alpha=0}^n = (t,W_t^1,\hdots, W_t^n)$. Let the drift Hamiltonian be $\mathcal{H}_0(g,m) = \frac{1}{2}Q(m,m) + V(g)$ where $V:G\to\mathbb{R}$ is the potential. Let $\{X_i\}_{i=1}^n$ be a basis for $\mathfrak{g}$ and take the diffusion Hamiltonians to be $\mathcal{H}_i(g,m) = -\sqrt{2\gamma}\, \mathrm{Tr}(gX_i)$ with $\gamma>0$. 

    The variational derivatives of the Hamiltonians are computed to be $\frac{\delta\mathcal{H}_0}{\delta m} = m$, $\frac{\delta\mathcal{H}_0}{\delta g} = \nabla V$, $\frac{\delta\mathcal{H}_i}{\delta m} = 0$ and $\frac{\delta\mathcal{H}_i}{\delta g} = -\sqrt{2\gamma}\,X_i$. For the double-bracket dissipation term, we compute $\{\mathcal{H}_0,\mathcal{H}_i\} = \sqrt{2\gamma} \,Q(m,X_i)$. Substituting this into \eqref{eq:dbdstandard} yields
    \begin{equation}
        \begin{aligned}
            \boldsymbol{d}m &= \big(-\beta\gamma m - (dL_{g^{-1}})_{g}\nabla V\big)\boldsymbol{d}t + \sqrt{2\gamma}\,X_i\,\boldsymbol{d}W^i_t\\
            \boldsymbol{d}g &= (dL_g)_e\,m\,\boldsymbol{d}t,
        \end{aligned}
    \end{equation}
    which one recognises as the usual kinetic Langevin equation with additive noise in the momentum. The basis elements $X_i$ turn the family of $\mathbb{R}$-Brownian motions into a Lie algebra Brownian motion. See \cite{kong2024convergence} for a direct proof of the fact that the invariant measure of momentum Langevin equation is the Gibbs measure.
\end{example}

\begin{example}[Momentum Langevin on Euclidean space]
Let $G=\mathbb{R}^n$ with the Euclidean metric and choose the semimartingale terms as $(Z^\alpha)^n _{\alpha=0} = (t,W_t^1,W_t^2,\hdots, W_t^n)$. Let the drift Hamiltonian be $\mathcal{H}_0(p_1,\hdots, p_n,q^1,\hdots, q^n) = \sum_{i=1}^n \frac{1}{2}p_i^2 + V(q^1,\hdots,q^n)$ and choose the diffusion Hamiltonians to be given by $\mathcal{H}_i(p_1,\hdots,p_n,q^1,\hdots, q^n) = U_i(q^1,\hdots, q^n)$ with $U_i:\mathbb{R}^n\to\mathbb{R}$ smooth. 

The partial derivatives of the Hamiltonians are given by $\frac{\partial\mathcal{H}_0}{\partial p_i} = p_i$, $\frac{\partial\mathcal{H}_0}{\partial q^i} = \frac{\partial V}{\partial q^i}$, $\frac{\partial\mathcal{H}_j}{\partial p_i} = 0$ and $\frac{\partial\mathcal{H}_j}{\partial q^i} = \frac{\partial U_j}{\partial q^i}$. Further, $\{\mathcal{H}_0,\mathcal{H}_j\} = -\sum_{i=1}^n p_i\frac{\partial U_j}{\partial q_i}$. Substituting this into \eqref{eq:dbdstandard} yields in vectorised form with $\boldsymbol{p}=(p_1,\hdots, p_n)$, $\boldsymbol{q} = (q^1,\hdots, q^n)$, $\boldsymbol{U} = (U_1,\hdots, U_n)$, and $\boldsymbol{W}_t = (W_t^1,\hdots, W_t^n)$ as
\begin{equation}
    \begin{aligned}
        \boldsymbol{dp} &= -\nabla_{\boldsymbol{q}} V\,\boldsymbol{d}t-\frac{\beta}{2}\nabla_{\boldsymbol{q}}\boldsymbol{U}(\nabla_{\boldsymbol{q}}\boldsymbol{U})^T\boldsymbol{p}\,\boldsymbol{d}t + \nabla_{\boldsymbol{q}} \boldsymbol{U}\circ \boldsymbol{dW}_t,\\
        \boldsymbol{dq} &= \boldsymbol{p}\, \boldsymbol{d}t.
    \end{aligned}
\end{equation}
Note that when $\boldsymbol{U}$ vanishes, one simply has deterministic symplectic dynamics as expected. Upon splitting the above equations into a deterministic Hamiltonian part and a space-dependent Langevin part, one obtains the Euclidean version of the Langevin Markov chain Monte Carlo algorithm presented in \cite{arnaudon2019irreversible}. Furthermore, choosing the diffusion Hamiltonians as $\mathcal{H}_i = -\sqrt{2\gamma}\,q_i$ yields
\begin{equation}
    \begin{aligned}
        \boldsymbol{dp} &= (-\beta\gamma\boldsymbol{p} - \nabla_{\boldsymbol{q}} V)\boldsymbol{d}t + \sqrt{2\gamma}\,\boldsymbol{dW}_t,\\
        \boldsymbol{dq} &= \boldsymbol{p}\,\boldsymbol{d}t,
    \end{aligned}
\end{equation}
which shows that on $\mathbb{R}^n$ double-bracket dissipation reproduces the kinetic Langevin equation. 
\end{example}

\begin{example}[Position Langevin on reductive Lie groups]
Let $G$ be a compact semisimple Lie group with bi-invariant metric $Q$. Choose the semimartingale terms as $(Z^\alpha)^n _{\alpha=0} = (t,W_t^1,W_t^2,\hdots, W_t^n)$. Let the drift Hamiltonian be $\mathcal{H}_0(g,m) = \frac{1}{2}Q(m,m) + V(g)$ where $V:G\to\mathbb{R}$ is a potential. Let $\{X_i\}_{i=1}^n$ be a basis for $\mathfrak{g}$ and take the diffusion Hamiltonians to be $\mathcal{H}_i(g,m) = \sqrt{2\gamma}\,Q(m,X_i)$. The variational derivatives can be computed to be $\frac{\delta\mathcal{H}_0}{\delta m} = m$, $\frac{\delta\mathcal{H}_0}{\delta g} = \nabla V$, $\frac{\delta\mathcal{H}_i}{\delta m} = \sqrt{2\gamma}\,X_i$ and $\frac{\delta\mathcal{H}_0}{\delta g} = 0$. Further, $\{\mathcal{H}_0,\mathcal{H}_i\} = \sqrt{2\gamma}\,Q(X_i,(dL_{g^{-1}})_g\nabla V)$. Substituting these expressions into \eqref{eq:dbdstandard} leads to
\begin{equation}
    \begin{aligned}
        \boldsymbol{d}m &= \big(-(dL_{g^{-1}})_g\nabla V+\beta\gamma Q(X_i,(dL_{g^{-1}})_g\nabla V)[X_i,m]\big)\,\boldsymbol{d}t + \sqrt{2\gamma}[X_i,m]\circ \boldsymbol{d}W_t^i,\\
        \boldsymbol{d}g &= (dL_g)_e\Big(\big(m - \beta\gamma \,\nabla V\big)\,\boldsymbol{d}t + \sqrt{2\gamma}\,X_i\, \boldsymbol{d}W_t^i\Big),
    \end{aligned}
\end{equation}
where $[X_i,m]$ measures the extent to which to momentum does not commute with the basis of the Lie algebra and $Q(X_i,(dL_{g^{-1}})_g\nabla V)X_i = \nabla V$ projects the gradient of the potential onto the Lie algebra basis. Note that on reductive Lie groups one can identify the coadjoint action with minus the adjoint action. Furthermore, the Lie-Poisson drift, i.e. $\mathrm{ad}^*_{\delta \mathcal{H}_0/\delta m}m$, is isospectral. Since the drift Hamiltonian is the kinetic energy associated with a bi-invariant metric, this term vanishes. Note also that the noise in the evolution of $g$ is precisely the Riemannian Brownian motion on a reductive Lie group with a bi-invariant metric that we constructed in Section \ref{sec:eem}.
\end{example}

\begin{example}[Position Langevin on Euclidean space]
Let $G=\mathbb{R}^n$ with the Euclidean metric and choose the semimartingale terms as $(Z^\alpha)^n _{\alpha=0} = (t,W_t^1,W_t^2,\hdots, W_t^n)$. Let the drift Hamiltonian be $\mathcal{H}_0(p_1,\hdots, p_n,q^1,\hdots, q^n) = \sum_{i=1}^n \frac{1}{2}p_i^2 + V(q^1,\hdots, q^n)$, where $V:\mathbb{R}^n\to\mathbb{R}$ is a potential function, and take the diffusion Hamiltonians as $\mathcal{H}_i(p_1,\hdots,p_n,q^1,\hdots, q^n) = \sqrt{2\gamma}\,p_i$. The partial derivatives of the Hamiltonians are given by $\frac{\partial\mathcal{H}_0}{\partial p_i} = p_i$, $\frac{\partial\mathcal{H}_0}{\partial q^i} = \frac{\partial V}{\partial q^i}$, $\frac{\partial \mathcal{H}_j}{\partial p_i} = \sqrt{2\gamma}\,\delta^i_j$ and $\frac{\partial\mathcal{H}_j}{\partial q^i} = 0$. Here $\delta^i_j$ denotes the Kronecker delta. Further, $\{\mathcal{H}_0,\mathcal{H}_i\} = -\sqrt{2\gamma}\frac{\partial V}{\partial q^i}$. Substituting these expressions into \eqref{eq:dbdstandard} yields in vectorised form with $\boldsymbol{p} = (p_1,\hdots, p_n)$, $\boldsymbol{q}=(q^1,\hdots, q^n)$ and $\boldsymbol{W}_t = (W_t^1,\hdots , W_t^n)$ as
\begin{equation}\label{eq:euclideandbd}
    \begin{aligned}
        \boldsymbol{dp} &= -\nabla_{\boldsymbol{q}} V\boldsymbol{d}t,\\
        \boldsymbol{dq} &= (\boldsymbol{p} - \beta\gamma\nabla_{\boldsymbol{q}} V)\boldsymbol{d}t + \sqrt{2\gamma}\,\boldsymbol{dW}_t.
    \end{aligned}
\end{equation}
This shows that stochastic symplectic dynamics on $\mathbb{R}^{2n}$ can be used to sample from distributions on $\mathbb{R}^n$ with dissipation determined by the potential.
\end{example}

\begin{example}[Symplectic Langevin on reductive Lie groups]
    Consider the stochastic equations of symplectic mechanics \eqref{eq:stochsymplectic} on the symplectic manifold $G\times\mathfrak{g}^*$ with $G$ an $n$-dimensional compact semisimple Lie group with bi-invariant metric $Q$. Choose the semimartingale terms as $(Z^\alpha)_{\alpha=0}^{2n} = (t, W_t^1, \hdots, W_t^n, \widetilde{W}_t^1,\hdots \widetilde{W}_t^n)$, where all Brownian motions are mutually independent and identically distributed. Let the drift Hamiltonian be $\mathcal{H}_0(g,m) = \frac{1}{2}Q(m,m) + V(g)$, where $V:G\to\mathbb{R}$ is a potential. Let $\{X_i\}_{i=1}^n$ be a basis for $\mathfrak{g}$ and take two families of diffusion Hamiltonians as $\mathcal{H}_i(g,m) = \sqrt{2\gamma_2}\,Q(m,X_i)$, and $\widetilde{\mathcal{H}}_j(g,m) = - \sqrt{2\gamma_1}\,\mathrm{Tr}(g X_j)$. Both $\gamma_1$ and $\gamma_2$ are positive constants. 
    
    The variational derivatives can be computed to be $\frac{\delta\mathcal{H}_0}{\delta m} =m$, $\frac{\delta\mathcal{H}_0}{\delta g}=\nabla V(g)$, $\frac{\delta\mathcal{H}_i}{\delta m} = \sqrt{2\gamma_2}\,X_i$, $\frac{\delta\mathcal{H}_i}{\delta g}=0$, $\frac{\delta\widetilde{\mathcal{H}}_i}{\delta m} = 0$, $\frac{\delta\widetilde{\mathcal{H}}_i}{\delta g}=-\sqrt{2\gamma_1}\,X_i$  For the double-bracket dissipation terms, we compute $\{\mathcal{H}_0,\widetilde{\mathcal{H}}_j\} = \sqrt{2\gamma_1} Q(m, X_j)$ and $\{\mathcal{H}_0,\mathcal{H}_i\} = \sqrt{2\gamma_2} Q(X_i,(dL_{g^{-1}})_g\nabla V)$. Hence the equations are given by
    
    \begin{equation}\label{eq:dbdsymp}
        \begin{aligned}
            \boldsymbol{d}m &= \big(- (dL_{g^{-1}})_g\nabla V -\beta\gamma_1 m +\beta\gamma_2 Q(X_i,(dL_{g^{-1}})_g\nabla V)[X_i,m]\big)\boldsymbol{d}t \\
            &\qquad+ \sqrt{2\gamma_1}\,X_i\,\boldsymbol{d}\widetilde{W}_t^i + \sqrt{2\gamma_2}[X_i,m]\circ\boldsymbol{d}W_t^i,\\ 
            \boldsymbol{d}g &= (dL_g)_e\Big(\big(m - \beta\gamma_2 \nabla V \big)\boldsymbol{d}t + \sqrt{2\gamma_2}\,X_i\,\boldsymbol{d}W_t^i\Big).
        \end{aligned}
    \end{equation}
    Note that $\gamma_1=0,\gamma_2\neq 0$ yields position Langevin and $\gamma_1\neq 0, \gamma_2=0$ yields momentum Langevin.
\end{example}

\begin{example}[Symplectic Langevin on Euclidean space]
    Let $G=\mathbb{R}^n$ with the Euclidean metric and choose the semimartingale terms as $(Z^\alpha)_{\alpha=0}^{2n} = (t,W_t^1,\hdots,W_t^n,\widetilde{W}_t^1,\hdots,\widetilde{W}_t^n)$, where all the Brownian motions are mutually independent and identically distributed. Let the drift Hamiltonian be $\mathcal{H}_0(p_1,\hdots,p_n,q^1,\hdots,q^n) = \sum_{i=1}^n\frac{1}{2}p_i^2 +V(q^1,\hdots,q^n)$, where $V:\mathbb{R}^n\to\mathbb{R}$ is a potential and take the diffusion Hamiltonians as $\mathcal{H}_i(p_1,\hdots,p_n,q^1,\hdots, q^n) = \sqrt{2\gamma_2}\,p_i$ and $\widetilde{\mathcal{H}}_j(p_1,\hdots,p_n,q^1,\hdots,q^n) = -\sqrt{2\gamma_1}\,q^j$, with $\gamma_1$ and $\gamma_2$ positive constants.

    The partial derivatives of the Hamiltonians are $\frac{\partial\mathcal{H}_0}{\partial p_i} = p_i$, $\frac{\partial\mathcal{H}_0}{\partial q^i} = \frac{\partial V}{\partial q^i}$, $\frac{\partial\mathcal{H}_j}{\partial p_i} = \sqrt{2\gamma_2}\,\delta^i_j$, $\frac{\partial\mathcal{H}_j}{\partial q^i} = 0$, $\frac{\partial\widetilde{\mathcal{H}}_j}{\partial p_i} = 0$, and $\frac{\partial\widetilde{\mathcal{H}}_j}{\partial q^i} = -\sqrt{2\gamma_1}\,\delta_{ij}$. The double bracket terms are given by $\{\mathcal{H}_0,\mathcal{H}_i\} = \sqrt{2\gamma_2}\,\frac{\partial V}{\partial q^i}$, $\{\mathcal{H}_0,\widetilde{\mathcal{H}}_i\} = \sqrt{2\gamma_1}\,p_i$. Hence the equations of motion in vectorised form with $\boldsymbol{p} = (p_1,\hdots, p_n)$, $\boldsymbol{q}=(q^1,\hdots, q^n)$, $\boldsymbol{W}_t = (W_t^1,\hdots , W_t^n)$ and $\widetilde{\boldsymbol{W}}_t = (\widetilde{W}_t^1,\hdots, \widetilde{W}_t^n)$ are given by
    \begin{equation}
        \begin{aligned}
            \boldsymbol{dp} &= (-\nabla V - \beta\gamma_1\boldsymbol{p})\boldsymbol{d}t + \sqrt{2\gamma_1}\,\boldsymbol{d}\widetilde{\boldsymbol{W}}_t,\\
            \boldsymbol{dq} &= (\boldsymbol{p} - \beta\gamma_2\nabla V)\boldsymbol{d}t + \sqrt{2\gamma_2}\,\boldsymbol{dW}_t.
        \end{aligned}
    \end{equation}
    This is Euclidean Langevin dynamics with noise in the entire phase space. A potential benefit of such dynamics is its enhanced exploration of phase space and its amenability to moment estimates.
\end{example}

\section{Conclusion}\label{sec:conclusion}
In this work, we recalled the notion of reductive Lie algebras and their connection to Riemannian geometry to obtain an explicit characterisation of Riemannian Brownian motion. Via Malliavin's transfer principle, we lifted deterministic curves describing symplectic mechanics to semimartingale-valued curves describing stochastic symplectic mechanics. Deterministic symplectic mechanics has Gibbs measures as the natural invariant measures, which are not preserved under Malliavin's transfer principle. By including double-bracket dissipation in addition to the stochastic perturbations, the Gibbs measures are recovered as the invariant measures, but now of the stochastic differential equations describing stochastic symplectic mechanics with structure-preserving dissipation. This result can be interpreted as a symplectic fluctuation-dissipation theorem. We then provide several examples of how one would use this framework on a reductive Lie group and in the Euclidean setting for kinetic Langevin dynamics. In the reductive Lie group case, a natural choice for diffusion Hamiltonian is to couple the noise to the momentum variable, which leads to Riemannian Brownian motion in the position variable. 

The symplectic framework naturally gives rise to three families of kinetic Langevin diffusions when the diffusion Hamiltonians are chosen to be linear in the Darboux coordinates:
\begin{itemize}
    \item \textbf{Momentum Langevin diffusion} – Here, noise and dissipation act only at the momentum level, corresponding to the standard kinetic Langevin diffusion.
    \item \textbf{Position Langevin diffusion} – In this case, additive noise and dissipation appear at the position level and multiplicative noise appears at the momentum level. In the position coordinate, the noise is Riemannian Brownian motion. The multiplicative noise vanishes if the momentum commutes with the basis vectors of the Lie algebra. The dissipation is determined by the potential and is generally nonlinear.
    \item \textbf{Symplectic Langevin diffusion} – This combines both momentum and position Langevin mechanisms, leading to a diffusion process where noise is introduced in the full phase space.
\end{itemize}
Further generalisations are possible by choosing different diffusion Hamiltonians, not necessarily linear in the Darboux coordinates. 

As was pointed out by \cite{kong2024convergence}, the degenerate nature of the noise for the momentum Langevin equation poses several analytical challenges. We expect that position Langevin will pose similar challenges. While the symplectic approach produces RBM on reductive Lie groups, we comment for Lie groups that are not reductive, the symplectic approach does not produce RBM since the Lie exponential does not coincide with the Riemannian exponential. By means of the symplectic approach, we are able to formulate Langevin dynamics with noise in the full phase space while maintaining the Gibbs measure as the invariant measure. This may aid future analyses of the symplectic Langevin model. 

We focused on formulating symplectic Langevin diffusions on reductive Lie groups because of its numerical benefits over general Riemannian manifolds. In future work, we plan to use the theoretical advantages that reductive Lie groups offer in simulations. In this simulation study, we will compare the various types of Langevin diffusions in their performance in sampling distributions and consider different types of numerical methods. The symplectic nature of the Langevin diffusions gives strong indication for the use of splitting methods and symplectic and/or energy preserving Lie group integrators. Two more directions of future research are, firstly, to extend the present setting to reductive homogeneous spaces, which include Stiefel manifolds, Grasmannians and flag manifolds and, secondly, the reductive Lie group setting includes also bi-invariant pseudo-metrics when one does not insist on compactness, which have negative curvature.

\section*{Acknowledgements}
We are grateful to Gergely Bodo, Sonja Cox, Darryl Holm, Alexander Lewis and Roy Schieven for many insightful discussions during the preparation of this work. EL was supported by NWO grant VI.Vidi.213.070. ODS acknowledges funding for a research fellowship from Quadrature Climate Foundation.

\bibliographystyle{plainnat}
\bibliography{biblio}

\end{document}